\documentclass[12pt]{article}
\usepackage{enumerate}
\usepackage{graphicx}      
\usepackage{natbib}
\usepackage{graphics}
\usepackage[section]{placeins}
\usepackage{float}
\usepackage{amsmath,amssymb}
\usepackage{delarray}
\usepackage{leftidx}
\usepackage{graphicx}
\usepackage{epsfig}
\usepackage{amsthm}
\usepackage{latexsym}
\usepackage{amsfonts}

\newtheoremstyle{theorem}
  {15pt}          
  {15pt}  
  {\sl}  
  {\parindent}
  {\sc}  
  {. }    
  { }    
  {}     


\def\1{\mathbf{1}}

\newtheoremstyle{defi}
  {15pt}          
  {15pt}  
  {\rm}  
  {\parindent}     
  {\sc}  
  {. }    
  { }    
  {}     
\theoremstyle{defi}

\usepackage{listings}
\usepackage{color}
\usepackage{amsthm}
\newtheorem{lem} {Lemma}

\newtheorem{thm} {Theorem}
\newtheorem{prop}{Proposition}

\newtheorem{corr}{Corollary}

\begin{document}


 \title{Well-posedness and regularity of the Cauchy problem for nonlinear fractional in time and space equations}

\author{V. N. Kolokoltsov$^1$\thanks{Partially supported by the IPI RAN grants RFBR 11-01-12026 and 12-07-00115,
and by the grant 4402 of the Ministry of Education and Science of Russia}, M. A. Veretennikova$^2$\thanks{Supported by EPSRC grant EP/HO23364/1 through MASDOC, University of Warwick, UK} \\ 
$^1$Department of Statistics, $^2$ Mathematics Institute \\
University of Warwick\\
Coventry, CV4 7AL, UK \\
$^1$v.kolokoltsov@warwick.ac.uk, $^2$m.veretennikova@warwick.ac.uk}

\maketitle

\begin{abstract}
The purpose is to study the Cauchy problem for non-linear in time and space pseudo-differential equations. These include the fractional in time versions of HJB equations governing the controlled scaled CTRW. As a preliminary step which is of independent interest we analyse the corresponding linear equation proving its well-posedness and smoothing properties.

\medskip


\end{abstract}


\section*{Introduction}

The purpose of this paper is to study well-posedness of the Cauchy problem for the fractional in time and space pseudo-differential equation
\begin{align}\label{H}
D^{* \beta}_{0,t}f(t,y)=-a(-\Delta)^{\alpha/2}f(t,y)+ H(t,y,\nabla f(t,y))
\end{align}

where $y \in \mathbb{R}^{d}, t \ge 0$, $\beta \in (0,1), \alpha \in (1,2]$, $H(t,y,p)$ is a Lipschitz function in all of its variables, and $f(0,y)=f_{0}(y)$ is known and bounded, and $a$ is a constant, $a > 0$. Here $\nabla$ denotes the gradient with respect to the spatial variable. For a function dependent on several spatial variables, say $x, y$, we may occasionally indicate the variable with respect to which the gradient is taken, by a subscript, $\nabla_{x}$. We denote by $D^{* \beta}_{0,t}$ the Caputo derivative:

\begin{equation}\label{Cap}
D^{* \beta}_{0,t}f(t,y)=\frac{1}{\Gamma(1-\beta)}\int_{0}^{t}\frac{d f(s,y)}{ds}(t-s)^{-\beta}ds,
\end{equation}

whilst $-(-\Delta)^{\alpha/2}$ is the fractional Laplacian
\begin{align}
-(-\Delta)^{\alpha/2}f(t,y)= C_{d, \alpha}\int_{\mathbb{R}^{d}}\frac{f(t,y) - f(t,x)}{|y-x|^{d + \alpha}}dx,
\end{align}
where $C_{d, \alpha}$ is a normalizing constant. Extension of our results for (\ref{H}) to the case where $H=H(t,y,f(t,y),\nabla f(t,y))$ is straightforward and we omit it here.


As a preliminary analysis we establish the regularity properties of the linear equations of the form
\begin{align}\label{h}
D_{0,t}^{* \beta}f(t,y)=-a(-\Delta)^{\alpha/2}f(t,y) + h(t,y),
\end{align}

with a given function $h$, an initial condition $f(0,y)=f_{0}(y)$, $\beta \in (0,1)$, $\alpha \in (1,2]$, and a constant $a > 0$. This allows one to reduce the analysis of (\ref{H}) to a fixed point problem. Section 3 is devoted to the linear problem (\ref{h}) and in section 4 we formulate and prove our main results for equation (\ref{H}).

\smallskip

In this section we present a literature review. Among researchers who studied solutions to fractional differential equations are \cite{mainardi}, \cite{meerschaert}, \cite{podlubny1999fractional}, \cite{bajlekova},  \cite{kexue}, \cite{lizama}, \cite{diethelm}, \cite{matar}, \cite{kochubei}, \cite{kilbas}, \cite{leonenko}, \cite{agarwal}. More results and reviews can be found in references therein. Fractional differential equations appear for example in modelling processes with memory, see \cite{uchaikin}, \cite{tarasov}, \cite{machado}, \cite{Escalas}.

Several authors solve fractional differential equations using Laplace transforms in time, see \cite{kexue}, \cite{kilbas} and \cite{lizama} for example. 

The book \cite{diethelm} covers analysis for Caputo time-fractional differential equations with the parameter $\beta > 0$, for example
\begin{align}
D^{* \beta}_{0,t}y(x)=-\mu y(x) + q(x),
\end{align}
with $y(0)=y_{0}^{(0)}$, $Dy(0)=y_{0}^{(1)}$, $\beta \in (1,2)$, $\mu > 0$.

In \cite{bajlekova} the theory for fractional differential equations in $L^{p}$ spaces is developed. Well-posedness of (\ref{h}) in $L^{p}$ may be deduced from there.

In \cite{meerschaert} the authors consider classical solutions for fractional Cauchy problems in bounded domains $D \in \mathbb{R}^{d}$ with Dirichlet boundary conditions.

In \cite{zhou} one may find the analysis for the non-local Cauchy problem in a Banach space, where instead of $-(-\Delta)^{\alpha/2}$ there is a general infinitesimal generator of a strongly continuous semigroup of bounded linear operators. The authors present conditions that need to hold to ensure existence of mild forms of the fractional differential equation.

The paper \cite{ma} establishes asymptotic estimates of solutions to the following fractional equation and its similar versions:
\begin{align}
D^{* \alpha}_{0,t}u(x,t)=a^{2}\frac{d^{2}u(x,t)}{dx^{2}}, 
\end{align}

for $t > 0, x \in \mathbb{R}$, $\alpha \in (0,1)$, $u(x,0)=\phi(x)$, $\lim_{|x| \rightarrow + \infty}u(x,t)=0$, however the case of the fractional Laplacian is not included and there is no $h(x,t)$ term on the right hand side (RHS).

In \cite{kokurin} the author studies the uniqueness of a solution to
\begin{align}
D^{* \alpha}_{0,t}u(t)=Au(t),
\end{align}
where $t>0, u(0)=u_{0}$, and $A$ is an unbounded closed operator in a Banach space, $\alpha \in (0,1)$. However there is no non-homogeneity term $h(t)$ on the RHS. For solvability of linear fractional differential equations in Banach spaces one may see \cite{gorenflo}, where

\begin{align}\label{lin}
D^{* \alpha}_{0,t}x(t)=Ax(t), \mbox{ for } m-1 < \alpha \le m \in \mathbb{N},
\end{align}
and $\frac{d^{k}}{dt^{k}}x(t)|_{t=0}=\xi_{k}$, for $k = 0, \ldots, m-1$. The authors give sufficient conditions under which the set of initial data $\xi_{k}$ for $k = 0, \ldots, m-1$ provides a solution to (\ref{lin}) of the form $\sum_{k=0}^{m-1}t^{k}E_{\alpha, k+1}(At^{\alpha})\xi_{k}$. In particular, these conditions depend on Roumieu, Gevrey and Beurling spaces related to the operator $A$. 

In \cite{tao} the authors use fixed point theorems to prove existence and uniqueness of a positive solution for the problem

\begin{align}
D^{\alpha}_{0,t}x(t)=f(t,x(t), -D^{\beta}_{0,t}x(t)), t \in (0,1),
\end{align}
with non-local Riemann-Stieltjes integral condition 
\begin{align}
D^{\beta}_{0,t}x(0)=D^{\beta + 1}_{0,t}x(0)=0, 
\end{align}
and $D^{\beta}_{0,t}x(1)=\int_{0}^{1}D^{\beta}_{0,t}x(s)dA(s)$, where $A$ is a function of bounded variation, $\alpha \in (2,3], \beta \in (0,1), \alpha - \beta > 2$. In \cite{tao} there are references to papers where fractional differential equations are inspected with the help of various fixed point theorems. Our analysis also includes a fixed point theorem, however its use and the problem itself are different from the one in \cite{tao}.

In \cite{kochubei} there is a construction and investigation of a fundamental solution for the Cauchy problem with a regularised fractional derivative $D^{\alpha}_{0,t, reg}$, and $\alpha \in (0,1)$ defined by

\begin{align}\label{reg}
D^{\alpha}_{0,t, reg}u(t,x)=\frac{1}{\Gamma(1-\alpha)}\left[\frac{\partial}{\partial t}\int_{0}^{t}(t-\tau)^{-\alpha}u(\tau, x)d\tau - t^{-\alpha}u(0,x) \right].
\end{align}
Note that
\begin{align}
D^{\alpha}_{0,t}u(t,x)=\frac{1}{\Gamma(1-\alpha)}\frac{\partial}{\partial t}\int_{0}^{t}(t-\tau)^{-\alpha}u(\tau, x)d\tau
\end{align}

is the definition of the Riemann-Liouville fractional derivative. Since $D^{* \alpha}_{0,t}f(t,x)= D^{\alpha}_{0,t}f(t,x) - \frac{t^{-\alpha}}{\Gamma(1-\alpha)}f(0,x)$, the regularised derivative in (\ref{reg}) is in fact identical to our definition of the Caputo derivative in (\ref{Cap}).

The problem studied in \cite{kochubei} is
\begin{align}
D^{\alpha}_{0, t, reg}u(t,x)-Bu(t,x)=f(t,x),
\end{align}
$t \in (0,T], x \in \mathbb{R}^{n}$, where 

\begin{align}
B = \sum_{i,j = 1}^{n}a_{ij}(x)\frac{\partial^{2}}{\partial x_{i} \partial x_{j}} + \sum_{j=1}^{n}b_{j}(x)\frac{\partial}{\partial x_{j}} + c(x)
\end{align}
with bounded real-valued coefficients. Our analysis goes beyond to include $B=-a(-\Delta)^{\alpha/2}$, with $a > 0$.

In \cite{pskhu} in particular the fundamental solution to the multi-time fractional differential equation
\begin{align}\label{psk}
\sum_{k=1}^{m}\lambda_{k}D^{* \beta_{k}}u(t,y) - \Delta_{x}u(t,y)=f(t,y),
\end{align}
is presented, for  $t = (t_{1}, \ldots, t_{n}) \in \mathbb{R}^{n}, y = (y_{1}, \ldots, y_{m}) \in \mathbb{R}^{m}$, and $\lambda = (\lambda_{1}, \ldots, \lambda_{m}) \in \mathbb{R}^{m}$, whilst $\Delta_{x}$ is the standard Laplacian operator and $\beta_{k} \in (0,1)$ for all $1\ge k \le m$.  There is also the proof of that the fundamental solution for (\ref{psk}) is unique. The uniqueness result covers a more broad range of fractional differential equations involving Dzhrbashyan-Nersesyan fractional in time differential equations. In our case there are fractional operators with respect to both spacial and temporal variables.

Denote a bounded domain by $D$. Taking $\alpha \in (0,2)$, $\beta \in (0,1)$ the paper \cite{chen} develops strong solutions to the equation
\begin{align}
D^{* \beta}_{0,t}u(t,x)=\Delta^{\alpha/2}_{x}u(t,x),
\end{align}

for $x \in D$, $t > 0$, $u(0,x)=f(x)$ for $x \in D$ and $u(t,x)=0$ for $x \in D^{c}$, $t > 0$.

Our approach to the non-linear FDE seems to be different and includes the fractional Laplacian $-(-\Delta)^{\alpha/2}$ instead of the standard one $\Delta_{y}$. We extend to the scenario with the RHS term including $H(t,y,\nabla f(t,y))$, although we concentrate on the case with only one fractional time derivative $D^{* \beta}_{0,t}$. 

\section{Regularity of linear fractional dynamics}
Our analysis of equation (\ref{h}) is based on the Fourier transform in space, where for a function $g(y)$ its Fourier transform will be defined in the following way

\begin{align}
\hat{g}(p)=\int_{R^{d}}e^{-ipy}g(y)dy.
\end{align}

Applying the Fourier transform to (\ref{h}) yields
\begin{align}
D^{* \beta}_{0,t}\hat{f}(t,p)=-a|p|^{\alpha}\hat{f}(t,p) + \hat{h}(t,p).
\end{align}






This is a standard linear equation with the Caputo fractional derivative. For continuous $h$ its solution is given by
\begin{align}\label{fh}
\hat{f}(t,p)=\hat{f}_{0}(p)E_{\beta,1}(-a|p|^{\alpha}t^{\beta}) + \int_{0}^{t}(t-s)^{\beta - 1}E_{\beta, \beta}(-a(t-s)^{\beta}|p|^{\alpha}))\hat{h}(s,p)ds,
\end{align}

where $E_{\beta, 1}$ and $E_{\beta, \beta}$ are Mittag-Leffler functions, see formulas $(7.3) - (7.4)$ in \cite{diethelm}. 

Let us recall that the Mittag-Leffler functions are defined for $Re(\beta) > 0$, and $\gamma, z \in \mathbb{C}$:

\begin{align}\label{e}
E_{\beta, \gamma}(z)=\sum_{i=1}^{\infty}\frac{z^{k}}{\Gamma(\beta k + \gamma)}.
\end{align}

We will use the following connection between $E_{\beta, \beta}$ and $E_{\beta, 1}$:
 
\begin{align}\label{betaone}
x^{\beta-1}E_{\beta, \beta}(-a|p|^{\alpha}x^{\beta})=-\frac{1}{a|p|^{\alpha}}\frac{d}{dx}E_{\beta,1}(-a|p|^{\alpha}x^{\beta}).
\end{align}
To prove (\ref{betaone}) one may use the representation of $E_{\beta,1}(-a|p|^{\alpha}x^{\beta})$ in (\ref{e}) and differentiate with respect to $x$ term by term. Now we present two convenient notations for further analysis. Let us denote 
\begin{align}\label{S}
S_{\beta, 1}(t,y)= \frac{1}{(2\pi)^{d}}\int_{\mathbb{R}^{d}}e^{ipy}E_{\beta, 1}(-a|p|^{\alpha}t^{\beta})dp
\end{align}

and

\begin{align}\label{Gb}
G_{\beta}(t,y)=\frac{t^{\beta-1}}{(2\pi)^{d}}\int_{\mathbb{R}^{d}}e^{ipy}E_{\beta,\beta}(-a|p|^{\alpha}t^{\beta})dp.
\end{align}

Using (\ref{betaone}) we can re-write it as

\begin{align}\label{gmitone}
G_{\beta}(t,y)=-\frac{1}{(2\pi)^{d}}\int_{\mathbb{R}^{d}}e^{ipy}\frac{1}{a|p|^{\alpha}}\frac{d}{dt}E_{\beta,1}(-a|p|^{\alpha}t^{\beta})dp.
\end{align} 

Applying the inverse Fourier transform to (\ref{fh}) 
we obtain:

\begin{align}\label{f}
f(t,y)=\int_{\mathbb{R}^{d}}S_{\beta,1}(t,y-x)f_{0}(x)dx + \int_{0}^{t} \int_{\mathbb{R}^{d}}G_{\beta}(t-s, y-x)h(s,x)dx ds.
\end{align}

It is natural to call this integral equation the mild form of the fractional linear equation (\ref{h}). In particular we see that the function $S_{\beta, 1}(t, y-y_{0})$ is the solution of equation (\ref{h}) with $f_{0}(y)= \delta(y-y_{0})$ and $h(t,y)=0$. On the other hand the function $G_{\beta}(t-t_{0}, y-y_{0})$ is the solution of (\ref{h}) with $f_{0}(y)=0$ and $h(t,y)=\delta(t-t_{0}, y-y_{0})$. Thus the functions $S_{\beta, 1}$ and $G_{\beta}$ may be called Green functions of the corresponding Cauchy problems. Notice the crucial difference with the usual evolution corresponding to $\beta = 1$ where $G_{\beta}$ and $S_{\beta,1}$ coincide.

In order to clarify the properties of $f$ in (\ref{f}) we are now going to carefully analyse the asymptotic properties of the integral kernels $S_{\beta, 1}(t,y)$ and $G_{\beta}(t,y)$.

\section{Asymptotic properties of $S_{\beta, 1}$ and $G_{\beta}$}

For $d \ge 1$ let us define the symmetric stable density $g$ in $\mathbb{R}^{d}$ as

\begin{align}\label{symm}
g(y; \alpha, \sigma, \gamma=0)=\frac{1}{(2\pi)^{d}}\int_{\mathbb{R}^{d}}\exp\{-ipy - a \sigma |p|^{\alpha}\}dp,
\end{align}

where $\alpha$ is the stability parameter, $\sigma$ is the scaling parameter and $\gamma$ is the skewness parameter which is $\gamma=0$ for symmetric stable densities. In $d=1$ and $\alpha \ne 1$ we define the fully skewed density with $\gamma = 1$ and without scaling:



\begin{align}\label{skew}
w(x; \alpha, 1) = \frac{1}{2\pi}Re \int_{-\infty}^{\infty}\exp\left\{-ipx - |p|^{\alpha}\exp\left\{-i\frac{\pi}{2}K(\alpha)\right\}\right\}dp, 
\end{align}
where $K(\alpha)=\alpha - 1 + \mbox{sign}(1-\alpha)$. The function $w(x; \alpha, 1)$ is infinitely differentiable and vanishes identically for $x <0$, see \cite{Zolotarev}, theorem C.3 and \textsection 2.2, equation (2.2.1a).

\smallskip

The starting point of the analysis of $S_{\beta, 1}, G_{\beta}$ is the following representation of the Mittag-Leffler function due to \cite{Zolotarev}, see chapter $2.10$, Theorem $2.10.2$, equations $(2.10.8 - 2.10.9)$. 
For $\beta \in (0,1)$ 
\begin{align}
E_{\beta,1}(-a\lambda)=\frac{1}{\beta}\int_{0}^{\infty}\exp(-a\lambda x)x^{-1-1/\beta}w(x^{-1/\beta}, \beta, 1)dx.
\end{align}

Substitute $\lambda = |p|^{\alpha}t^{\beta}$:
\begin{align}\label{E}
E_{\beta,1}(-a|p|^{\alpha}t^{\beta})=\frac{1}{\beta}\int_{0}^{\infty}\exp(-a|p|^{\alpha}t^{\beta}x)x^{-1-1/\beta}w(x^{-1/\beta}, \beta, 1)dx.
\end{align}

So then

\begin{align}\label{Q1}
t^{\beta-1}E_{\beta, \beta}(-a|p|^{\alpha}t^{\beta})=\frac{-1}{a|p|^{\alpha}}\frac{d}{dt}E_{\beta,1}(-a|p|^{\alpha}t^{\beta}) \nonumber \\
=t^{\beta-1}\int_{0}^{\infty}x^{-1/\beta}\exp(-a|p|^{\alpha}t^{\beta}x) w(x^{-1/\beta}, \beta, 1)dx,
\end{align}

implying

\begin{align}\label{gbb}
G_{\beta}(t,y)=\frac{1}{(2\pi)^{d}}\int_{\mathbb{R}^{d}}e^{ipy}E_{\beta, \beta}(-a|p|^{\alpha}t^{\beta})t^{\beta - 1}dp \nonumber \\
=\frac{t^{\beta - 1}}{(2\pi)^{d}}\int_{0}^{\infty}\int_{\mathbb{R}^{d}}e^{ipy}\exp\{-a|p|^{\alpha}t^{\beta}x\}x^{-1/\beta} w(x^{-1/\beta}, \beta, 1)dp dx \nonumber \\
=t^{\beta - 1}\int_{0}^{\infty}x^{-1/\beta}w(x^{-1/\beta}, \beta, 1)g(-y, \alpha, t^{\beta}x)dx,
\end{align}

where $g$ is as in (\ref{symm}) and $w$ is as in (\ref{skew}).

Throughout this paper we shall denote by $C$ various constants that may be different from formula to formula and line to line.

\begin{thm}\label{th66}
For $\beta \in (0,1)$

\begin{align}\label{ggg}
\int_{\mathbb{R}^{d}} |G_{\beta}(t,y)|dy \le C t^{\beta - 1},
\end{align}
where $C > 0$ is a constant.
\end{thm}

\begin{proof}
Let us split the integral representing $G_{\beta, 1}(t,y)$ in the sum of two, so that

\begin{align}
G_{\beta}(t,y)= I_{A} + I_{B}, 
\end{align}
where
\begin{align}
I_{A}=\frac{t^{\beta - 1}}{(2\pi)^{d}}\int_{|y|^{\alpha}t^{-\beta}}^{\infty}\int_{\mathbb{R}^{d}}e^{ipy}\exp\{-a|p|^{\alpha}t^{\beta}x\}x^{-1/\beta}w(x^{-1/\beta}, \beta, 1)dp dx \nonumber \\
=t^{\beta - 1}\int_{|y|^{\alpha}t^{-\beta}}^{\infty}x^{-1/\beta}w(x^{-1/\beta}, \beta, 1)g(-y, \alpha, t^{\beta}x)dx
\end{align}
and
\begin{align}
I_{B}=\frac{t^{\beta - 1}}{(2\pi)^{d}}\int_{0}^{|y|^{\alpha}t^{-\beta}}\int_{\mathbb{R}^{d}}e^{ipy}\exp\{-a|p|^{\alpha}t^{\beta}x\}x^{-1/\beta}w(x^{-1/\beta}, \beta, 1)dp dx \nonumber \\
t^{\beta - 1}\int_{0}^{|y|^{\alpha}t^{-\beta}}x^{-1/\beta}w(x^{-1/\beta}, \beta, 1)g(-y, \alpha, t^{\beta}x)dx.
\end{align}



To estimate $|I_{A}|$ and $|I_{B}|$, let us examine cases $|y| > t^{\beta/\alpha}$ and $|y| \le t^{\beta/\alpha}$ and start with $|y| > t^{\beta/\alpha}$. Note that the asymptotic expansions for $g(y, \alpha, \sigma)$ and $g(-y, \alpha, \sigma)$, namely, (\ref{zero}) and (\ref{infinity}) appearing in the Appendix, are the same, by inspection. Since $x > |y|^{\alpha}t^{-\beta}$ in $I_{A}$ we may use the asymptotic for $|y|/x^{1/\alpha}
t^{\beta/\alpha} \rightarrow 0$, see (\ref{zero}). We also use that for $x \rightarrow \infty$, $x^{-1/\beta} \rightarrow 0$, so for $x \rightarrow \infty$ we have $w(x^{-1/\beta}, \beta, 1) \sim C$, where $C \ge 0$ is a constant. Thus we have 

\begin{align}\label{zz}
|I_{A}| \le \Bigg| \int_{|y|^{\alpha}t^{-\beta}}^{\infty} x^{-1/\beta - d/\alpha}w(x^{-1/\beta}, \beta, 1)A_{0}t^{\beta - 1 - d\beta/\alpha}dx \Bigg| \nonumber \\
\le C t^{\beta - 1 - d \beta/\alpha}|A_{0}|\frac{(|y|^{\alpha}t^{-\beta})^{1-1/\beta- d/\alpha}}{|1 - 1/\beta - d/\alpha|} \nonumber \\
\le C t^{\beta - 1 - d\beta/\alpha}\frac{|A_{0}|}{|1 - 1/\beta - d/\alpha|}(|y|^{\alpha}t^{-\beta})^{1- 1/\beta - d/\alpha} \nonumber \\
\le C t^{\beta - 1 + 1 - \beta}|y|^{\alpha - \alpha/\beta - d}.
\end{align} 

Now, let's study $I_{B}$ in case $|y| > t^{\beta/\alpha}$. Here we use the asymptotic expansion                                                                                                                                for $|y|/x^{1/\alpha}t^{\beta/\alpha} \rightarrow \infty$ as it appears in (\ref{infinity}) in the Appendix and take the first term only. Here we use the change of variables $z=x^{-1/\beta}$.



\begin{align}
|I_{B}| \le \Bigg|A_{1}t^{\beta - 1}\int_{0}^{|y|^{\alpha}t^{-\beta}}x^{-1/\beta}w(x^{-1/\beta}, \beta, 1)|y|^{-d-\alpha}t^{\beta} x dx \Bigg| \nonumber \\
\le Ct^{2\beta - 1}|y|^{-d-\alpha}\int_{|y|^{-\alpha/\beta}t}^{\infty}z^{-2\beta}w(z, \beta, 1)dz.
\end{align}

We split this integral into two parts: $z \in [1,\infty)$ and $z \in (|y|^{-\alpha/\beta}t, 1)$. Firstly,

\begin{align}
t^{2\beta - 1}|y|^{-d-\alpha}\int_{1}^{\infty}z^{-2\beta}w(z, \beta, 1)dz \nonumber \\
 \le t^{2\beta - 1}|y|^{-d-\alpha}\int_{0}^{\infty} w(z, \beta, 1)dz \le Ct^{2\beta - 1}|y|^{-d-\alpha},
\end{align}



In case $z \in (|y|^{-\alpha/\beta}t, 1)$ we may use that $z$ is small and so $z^{-2\beta}w(z, \beta, 1) < Cz^{-2\beta + q - 3}$, for any $q > 1$. So

\begin{align}\label{ww}
t^{2\beta - 1}|y|^{-d-\alpha}\int_{|y|^{-\alpha/\beta}t}^{1}z^{-2\beta}w(z, \beta, 1)dz 
\le Ct^{2\beta - 1}|y|^{-d-\alpha}\int_{|y|^{-\alpha/\beta}}^{1}z^{-2\beta + q - 3}dz \nonumber \\
= C t^{2\beta - 1}|y|^{-d-\alpha} \left(1 - (|y|^{-\alpha/\beta}t)^{-2\beta + q - 2} \right).
\end{align}




\medskip

Now let's study the case $|y| \le t^{\beta/\alpha}$. For $I_{A}$ we use that $x$ is large, so $x^{-1/\beta}$ is small, and that for $q \ge 4$ we have $x^{-d/\alpha - 1/\beta}w(x^{-1/\beta}) < Cx^{-d/\alpha- (\frac{q-2}{\beta})}$. Here $|y|^{\alpha} \le t^{\beta}$ and we obtain

\begin{align}\label{yy}
|I_{A}| \le Ct^{\beta - 1 - d\beta/\alpha}\Bigg|\int_{|y|^{\alpha}t^{-\beta}}^{\infty}A_{0}x^{-d/\alpha - \frac{q - 2}{\beta}}dx \Bigg| \nonumber \\
\le t^{\beta - 1 - \beta d/\alpha}\left(y^{\alpha}t^{-\beta}\right)^{-d/\alpha - \frac{q - 2}{\beta} + 1}C \nonumber \\
\le t^{\beta - 1 - d\beta/\alpha}t^{-d\beta/\alpha - (q - 2) + \beta + d\beta/\alpha + (q-2) - \beta} C \nonumber \\
\le t^{\beta - 1 - d\beta/\alpha} C.
\end{align}

As for $I_{B}$ in case $|y| \le t^{\beta/\alpha}$,

\begin{align}\label{xx}
|I_{B}| \le C \int_{0}^{|y|^{\alpha}t^{-\beta}}x^{1-1/\beta}w(x^{-1/\beta}, \beta, 1)t^{2\beta - 1}|y|^{-d-\alpha}dx \nonumber \\
\le C |y|^{-d-\alpha}t^{2\beta - 1}\int_{0}^{|y|^{\alpha}t^{-\beta}}x^{1-1/\beta}(x^{-1/\beta})^{-1-\beta}dx 
\le C |y|^{2\alpha - d}t^{-\beta - 1}.
\end{align}

Integrating (\ref{zz}) in polar coordinates gives

\begin{align}\label{1}
\int_{|y|>t^{\beta/\alpha}}|I_{A}|dy \le C \int_{|r| > t^{\beta/\alpha}}|r|^{\alpha - \alpha/\beta - d + d - 1}d|r| \le (t^{\beta/\alpha})^{\alpha - \alpha/\beta}C = t^{\beta - 1}C,
\end{align}

Integration of (\ref{ww}) in polar coordinates gives

\begin{align}\label{2}
\int_{|y|>t^{\beta/\alpha}}|I_{B}|dy \le C t^{2\beta - 1}\int_{|r| > t^{\beta/\alpha}}|r|^{-d-\alpha + d - 1}dr \nonumber \\
+C t^{2\beta - 1}\int_{|r| > t^{\beta/\alpha}}|r|^{d-1-d-\alpha}|r|^{2\alpha}t^{-2\beta}dr \nonumber \\
=Ct^{2\beta - 1}(t^{\beta/\alpha})^{-\alpha} + C t^{2\beta - 1 - 2\beta}(t^{\beta/\alpha})^{\alpha} = Ct^{\beta - 1}.
\end{align}

Integration of (\ref{yy}) gives

\begin{align}\label{3'}
\int_{|y| \le t^{\beta/\alpha}}|I_{A}| dy
\le Ct^{\beta - 1 -d\beta/\alpha}\int_{|r| \le t^{\beta/\alpha}}|r|^{d-1}d|r| \nonumber \\
\le t^{d\beta/\alpha - d\beta/\alpha + \beta - 1} \frac{C|A_{0}|}{|d|} 
\le t^{\beta - 1}\frac{|A_{0}|C}{d}.
\end{align}

Integration of (\ref{xx}) yields

\begin{align}\label{4}
\int_{|y| \le t^{\beta/\alpha}}|I_{B}|dy \le C\int_{|r|\le t^{\beta/\alpha}}t^{-\beta - 1}|r|^{-d+2\alpha + d - 1}dr  
\le C t^{-\beta - 1}(t^{\beta/\alpha})^{2\alpha} = C t^{\beta - 1}.
\end{align}

Combining (\ref{1})--(\ref{4}) yields (\ref{ggg}).





\end{proof}

\begin{thm}\label{th67}
For $\beta \in (0, 1)$ and for $\alpha \in (1,2)$
\begin{align}\label{thirty}
\int_{\mathbb{R}^{d}} |\nabla G_{\beta}(t,y)| dy\le t^{\beta - 1 - \beta/\alpha}C.
\end{align}
\end{thm}
\begin{proof}
In case $|y| > t^{\beta/\alpha}$, we have $|y|^{-1}< t^{-\beta/\alpha}$ and so differentiation with respect to $y$ yields

\begin{align}\label{se}
|\nabla I_{A}| \le C t^{-\beta/\alpha}|I_{A}|
\end{align}

and
\begin{align}\label{secF2}
|\nabla I_{B}| \le C t^{-\beta/\alpha}|I_{B}|.
\end{align}

In case $|y| \le t^{\beta/\alpha}$ we need to take into account the second term of the asymptotic expansion, since the first term is independent of $|y|$. Consequently,

\begin{align}\label{first}
| \nabla I_{A}| \le C \int_{|y|^{\alpha}t^{-\beta}}^{\infty}x^{-d/\alpha}t^{-d\beta/\alpha}|y|(xt^{\beta})^{-2/\alpha}x^{-1/\beta}t^{\beta - 1}dx \nonumber \\
\le C \int_{|y|^{\alpha}t^{-\beta}}^{\infty}x^{-d/\alpha - 2/\alpha - 1/\beta}t^{-d\beta/\alpha - 2\beta/\alpha + \beta - 1}dx \nonumber \\
\le C t^{-d\beta/\alpha - 2\beta/\alpha + \beta - 1}|y| - C(|y|^{\alpha}t^{-\beta})^{-d/\alpha - 2/\alpha -1/\beta + 1}t^{\beta/\alpha} \nonumber \\
=C t^{-d\beta/\alpha - 2\beta/\alpha + \beta - 1}|y| - C t^{\beta/\alpha}|y|^{-d-2-\alpha/\beta + \alpha}.
\end{align}

Integration of the first term in (\ref{first}) yields

\begin{align}\label{co1}
C \int_{|r| < t^{\beta/\alpha}} t^{-d\beta/\alpha - 2\beta/\alpha + \beta - 1}|r|^{d - 1 + 1}dr \nonumber \\
\le Ct^{-d\beta/\alpha - \beta/\alpha + \beta - 1 + d\beta/\alpha} 
\le C t^{\beta - 1 - \beta/\alpha}.
\end{align}

Integration of the second term in (\ref{first}) gives

\begin{align}\label{co2}
\int_{|y| < t^{\beta/\alpha}}t^{\beta/\alpha}|y|^{-d+d-3-\alpha/\beta + \alpha}dy \nonumber \\
\le t^{\beta/\alpha}(t^{\beta/\alpha})^{-2-\alpha/\beta + \alpha} \le t^{\beta - 1 - \beta/\alpha}.
\end{align} 

Combining (\ref{co1}) and (\ref{co2})

\begin{align}\label{secA}
\int_{|y| \le t^{\beta/\alpha}}| \nabla I_{A} | dy \le Ct^{\beta - 1 - \beta/\alpha}.
\end{align}

As for $I_{B}$ for $|y| \le t^{\beta/\alpha}$

\begin{align}
\left|\nabla I_{B} \right| \le C t^{2\beta - 1}|y|^{-d-\alpha - 1}\int_{0}^{|y|^{\alpha}t^{-\beta}}\xi^{1-1/\beta}w(\xi^{-1/\beta}, \beta, 1)d\xi \nonumber \\
\le C t^{2\beta - 1}|y|^{-d-\alpha - 1}\int_{0}^{|y|^{\alpha}t^{-\beta}}\xi^{1-1/\beta}(\xi^{-1/\beta})^{-1-\beta}d\xi \nonumber \\
\le Ct^{2\beta - 1}|y|^{-d-\alpha - 1}(|y|^{\alpha}t^{-\beta})^{3}
\le Ct^{-\beta - 1}|y|^{-d+2\alpha - 1}.
\end{align}

Integration gives

\begin{align}
\int_{|y| \le t^{\beta/\alpha}}| \nabla I_{B}|dy \le C \int_{|y| \le t^{\beta/\alpha}}t^{-\beta - 1}|y|^{-d+d+2\alpha - 2}dy  \le C t^{-\beta - 1 - \beta/\alpha}.
\end{align}

So 
\begin{align}\label{sec1}
\int_{|y| \le t^{\beta/\alpha}}| \nabla I_{B}| dy \le C t^{\beta - 1 - \beta/\alpha}.
\end{align}

Since 
\begin{align}
\int_{\mathbb{R}^{d}}|\nabla G_{\beta}(t,y)| dy \le \int_{\mathbb{R}^{d}}|\nabla I_{A}|dy + \int_{\mathbb{R}^{d}}|\nabla I_{B}| dy 
\end{align}

combining results (\ref{se}), (\ref{secF2}), (\ref{secA}) and (\ref{sec1}) we obtain

\begin{align}
\int_{\mathbb{R}^{d}}|\nabla G_{\beta}(t,y)| dy \le C t^{\beta - 1 - \beta/\alpha}.
\end{align} 

which proves (\ref{thirty}).

\end{proof}

\medskip

Now let's consider the case $\alpha = 2$.

\begin{thm}\label{th68}
Let $G_{\beta, 1}(t,y)$ be as in (\ref{Gb}) and (\ref{gbb}).
For $\alpha = 2$ and any $\beta \in (0,1)$:
\begin{itemize}
\item $\int_{0}^{t}\int_{\mathbb{R}^{d}}|G_{\beta}(t,y)|dy ds = O(t^{\beta})$,
\item $\int_{0}^{t}\int_{\mathbb{R}^{d}}| \nabla G_{\beta}(t,y)|dy ds = O(t^{\beta/2})$.
\end{itemize}

\end{thm}

\begin{proof}

Note that

\begin{align}
\int_{\mathbb{R}^{d}}\exp\{-a \sigma p^{2}-iyp\}dp = \left(\frac{\sqrt{\pi}}{\sqrt{\sigma}}\right)^{d}\exp\left\{-\frac{y^{2}}{4a\sigma}\right\},
\end{align}

where in our case $\sigma = x t^{\beta}$. Substitute this into (\ref{gbb}) to obtain
\begin{align}
G_{\beta}(t,y)=\frac{t^{\beta - 1}}{(2\pi)^{d}}\int_{0}^{\infty}x^{-1/\beta}w(x^{-1/\beta}, \beta, 1)\left(\frac{\sqrt{\pi}}{\sqrt{t^{\beta}x}}\right)^{d}\exp\left\{\frac{-y^{2}}{4at^{\beta}x}\right\}dx
\end{align}


where $y^{2} = y_{1}^{2} + y_{2}^{2} + \ldots + y_{d}^{2}$. We are interested in $\int_{0}^{t}\int_{\mathbb{R}^{d}}|G_{\beta}(t,y)|dy ds$.
Integrating $y$-dependent terms in $G_{\beta}$ with respect to $y$ gives

\begin{align}\label{using1}
\int_{\mathbb{R}^{d}}\exp\left\{-|y|^{2}/4axt^{\beta}\right\}dy = (4\pi x t^{\beta})^{d/2}= C x^{d/2}t^{\beta d/2}.
\end{align}

The term $x^{d/2}t^{\beta d/2}$ cancels out with $\left(\frac{1}{\sqrt{t^{\beta}x}}\right)^{d}$ and we obtain

\begin{align}
\int_{\mathbb{R}^{d}}|G_{\beta}(t,y)|dy = I(t) 
= C \int_{0}^{\infty} x^{-1/\beta}w(x^{-1/\beta}, \beta, 1)t^{\beta - 1}dx.
\end{align}
Now we split the integral $I(t)$ into $2$ parts: $I_{a}(t)$ for $x>1$ and $I_{b}(t)$ for $0 \le x \le 1$. In $I_{a}(t)$, $x > 1$ and so $x^{-1/\beta} < 1$ and $w(x^{-1/\beta}, \beta, 1) \sim C$, so we have

\begin{align}\label{Ia}
I_{a}(t) = \int_{1}^{\infty}Ct^{\beta - 1}x^{- 1/\beta}w(x^{-1/\beta}, \beta, 1)dx \nonumber \\ 
\le t^{\beta - 1}\int_{1}^{\infty}x^{-1/\beta}Cdx 
\le C1^{1-1/\beta}t^{\beta - 1} = Ct^{\beta - 1}.
\end{align}

Integrating with respect to $s$ gives

\begin{align}\label{bet}
\int_{0}^{t} |I_{a}(t-s)| ds \le \int_{0}^{t}C(t-s)^{\beta - 1}ds = C t^{\beta}.
\end{align}

For $I_{b}(t)$, $x \le 1$, so $x^{-1/\beta} \ge 1$ and $w(x^{-1/\beta}, \beta, 1) \sim (x^{-1/\beta})^{-1-\beta} = x^{1 + 1/\beta}$ and

\begin{align}
I_{b}(t) = \int_{0}^{1}Cx^{-1/\beta}w(x^{-1/\beta}, \beta, 1)t^{\beta - 1}dx \nonumber \\
\le Ct^{\beta - 1}\int_{0}^{1}x^{-1/\beta + 1/\beta + 1}dx \le C.
\end{align}
with a constant $C_{2} > 0$. Now we integrate with respect to $s$

\begin{align}
\int_{0}^{t}|I_{b}(t-s)|ds = \int_{0}^{t}C(t-s)^{\beta - 1}ds = C t^{\beta}.
\end{align}

Together with (\ref{Ia}) and (\ref{bet}) this yields the first statement of the theorem.

Differentiating $G_{\beta}$ with respect to $y$ gives us

\begin{align}
I_{1}(t)=\int_{\mathbb{R}^{d}}| \nabla G_{\beta}(t,y)|dy \nonumber \\
=\int_{0}^{\infty}\int_{\mathbb{R}^{d}}t^{\beta - \beta - 1 - \beta d/2}x^{-1-1/\beta-d/2}|y|\exp\{-|y|^{2}/4axt^{\beta}\}w(x^{-1/\beta}, \beta, 1)dy dx.
\end{align}

Since

\begin{align}
\int_{\mathbb{R}^{d}}|y|\exp\{-|y|^{2}/4axt^{\beta}\}dy = C x t^{\beta} (\sqrt{xt^{\beta}})^{d-1} 
= C x^{\frac{d+1}{2}}t^{\frac{\beta (d + 1)}{2}},
\end{align}

we have

\begin{align}
I_{1}(t)=\int_{0}^{\infty}\int_{\mathbb{R}^{d}}t^{\beta - \beta - 1 - \beta d/2}x^{-1-1/\beta-d/2}|y|\exp\{-|y|^{2}/4axt^{\beta}\}w(x^{-1/\beta}, \beta, 1)dy dx \nonumber \\
= C \int_{0}^{\infty}t^{-1 + \beta/2}x^{-1/2-1/\beta}w(x^{-1/\beta}, \beta, 1)dx.
\end{align}

Now we split the integral $I_{1}(t)$ into parts corresponding to $x \in (0,1)$ and $x \in [1, \infty)$: 
\begin{align}
I_{2}(t)=\int_{0}^{1}t^{-1 + \beta/2}x^{-1/2-1/\beta}w(x^{-1/\beta}, \beta, 1)dx
\end{align}

and
\begin{align}
I_{3}(t)=\int_{1}^{\infty}t^{-1 + \beta/2}x^{-1/2-1/\beta}w(x^{-1/\beta}, \beta, 1)dx.
\end{align}

Let's examine $I_{2}(t)$. Since $x \in (0,1)$, we have $w(x^{-1/\beta}, \beta, 1) \sim (x^{-1/\beta})^{-1-\beta}$, so

\begin{align}
I_{2}(t) = \int_{0}^{1}t^{-1+\beta/2}x^{-1/2 - 1/\beta}w(x^{-1/\beta}, \beta, 1)dx \nonumber \\
=\int_{0}^{1}t^{-1+\beta/2}x^{-1/2 - 1/\beta + 1 + 1/\beta}dx = 2t^{-1+\beta/2}/3.
\end{align}

Integrating

\begin{align}
\int_{0}^{t}|I_{2}(t-s)|ds \le \int_{0}^{t}(t-s)^{-1+\beta/2}ds = t^{\beta/2}.
\end{align}

Now, for $I_{3}(t)$ we use that $x^{-1/\beta} \le 1$ and so $w(x^{-1/\beta}, \beta, 1) \sim C$.

\begin{align}
|I_{3}(t)| \le \Bigg|\int_{1}^{\infty}t^{-1+\beta/2}x^{-1/2 -1/\beta}w(x^{-1/\beta}, \beta, 1)dx\Bigg| \nonumber \\
\le C t^{-1+\beta/2}\Bigg|\int_{1}^{\infty}x^{-1/2 - 1/\beta}dx\Bigg| = Ct^{-1+\beta/2}.
\end{align}
Integrating with respect to $s$

\begin{align}
\int_{0}^{t}|I_{3}(t-s)|ds \le \int_{0}^{t}(t-s)^{\beta/2 - 1}ds = C t^{\beta/2}.
\end{align}

Note that $\beta/2 = \beta - \beta/\alpha$ for $\alpha = 2$. So for $\alpha = 2$ the form of the estimate is the same as for $\alpha \in (1,2)$.

\end{proof}

The following corollary is a consequence of the previous theorem.

\begin{corr}
For $\alpha = 2$ and $\beta \in (0,1)$

\begin{align}\label{re}
\int_{0}^{t}\int_{\mathbb{R}^{d}}\left(| \nabla G_{\beta}(t,y)| + |G_{\beta}(t,y)| \right)dy ds = O(t^{\beta/2}).
\end{align}
\begin{proof}
Since $\beta/2 < \beta$, we take the minimum power, $\beta/2$, to write the common estimate of the terms $\int_{\mathbb{R}^{d}}| \nabla G_{\beta}(t,y)| dy$ and $\int_{\mathbb{R}^{d}} |G_{\beta}(t,y)| dy$, obtaining
\begin{align}
\int_{\mathbb{R}^{d}}\left(| \nabla G_{\beta}(t,y)| + |G_{\beta}(t,y)| \right) dy = O(t^{\beta/2 - 1}),
\end{align}

substitute $t$ by $t-s$ and we use that
\begin{align}
\int_{0}^{t}(t-s)^{\beta/2 - 1}ds = C t^{\beta/2},
\end{align}
which yields the result (\ref{re}).
\end{proof}
\end{corr}

\medskip

Here we present several theorems regarding $S_{\beta, 1}(t,y)$ which are particularly useful for the well-posedness analysis of (\ref{h}) and (\ref{H}).

\begin{thm}\label{th69}
The first term from the RHS of (\ref{h}) satisfies
\begin{align}\label{b}
\left|\int_{\mathbb{R}^{d}}S_{\beta, 1}(t,y-x)f_{0}(x)dx\right| = O(t^{0}).
\end{align}
\end{thm}


\begin{proof}
Using (\ref{S}) and (\ref{E}) we represent $S_{\beta,1}(t,y)$ as

\begin{align}\label{sbone}
I=\frac{1}{\beta (2\pi)^{d}}\int_{\mathbb{R}^{d}}\int_{0}^{\infty}e^{ipy}\exp\{-a|p|^{\alpha}t^{\beta}\xi\}\xi^{-1-1/\beta}w(\xi^{-1/\beta}, \beta, 1)d\xi dp
\end{align}

and use the assumption $|f_{0}(y)| < C$. We split the integral $I$ into two parts: $I_{A}$ for $\xi \in [|y|^{\alpha}t^{-\beta}, \infty)$ and $I_{B}$ for $\xi \in (0, |y|^{\alpha}t^{\beta})$. There are $2$ cases for each of the integrals: $|y| \le t^{\beta/\alpha}$ and $|y| > t^{\beta/\alpha}$. Let us study $|I_{B}|$ in the case $|y| \le t^{\beta/\alpha}$.

\begin{align}
|I_{B}| \le C\int_{0}^{|y|^{\alpha}t^{-\beta}}\xi^{-1-1/\beta}w(\xi^{-1/\beta}, \beta, 1)|y|^{-d-\alpha}t^{\beta}\xi d\xi \nonumber \\
\le C \int_{0}^{|y|^{\alpha}t^{-\beta}}\xi^{-1/\beta}(\xi^{-1/\beta})^{-1-\beta}t^{\beta}|y|^{-d-\alpha}d\xi \nonumber \\
\le C \int_{0}^{|y|^{\alpha}t^{-\beta}}\xi t^{\beta}|y|^{-d-\alpha} d\xi \nonumber \\
=  C (|y|^{\alpha}t^{-\beta})^{2}|y|^{-d-\alpha}t^{\beta} = C t^{-\beta}|y|^{-d + \alpha}.
\end{align}
Now, integrating gives

\begin{align}
\int_{|y| \le t^{\beta/\alpha}}|I_{B}|dy \le C\int_{|y| \le C t^{\beta/\alpha}}t^{-\beta}|y|^{-d + \alpha + d - 1}dy = C t^{-\beta}(t^{\beta/\alpha})^{\alpha} = O(t^{0}).
\end{align}

Let's study $|I_{B}|$ in case $|y| > t^{\beta/\alpha}$. Here we split the integral $I_{B}$ into 2 parts: when $\xi \in (0, 1]$ and when $\xi \in (1, |y|^{\alpha}t^{-\beta})$.

\begin{align}
|I_{B}| \le C \int_{0}^{|y|^{\alpha}t^{-\beta}}\xi^{-1/\beta}w(\xi^{-1/\beta}, \beta, 1)t^{\beta}|y|^{-d-\alpha}d\xi,
\end{align}

so since for $\xi \le 1$, $w(\xi^{-1/\beta}, \beta, 1) \sim (\xi^{-1/\beta})^{-1-\beta}$, we have

\begin{align}
\int_{0}^{1}\xi^{-1/\beta}w(\xi^{-1/\beta}, \beta, 1)t^{\beta}|y|^{-d-\alpha}d\xi
\le C |y|^{-d-\alpha}t^{\beta}.
\end{align}

Integration yields
\begin{align}\label{l3}
\int_{|y| > t^{\beta/\alpha}}t^{\beta}|y|^{-d-\alpha + d - 1}dy = t^{\beta}(t^{\beta/\alpha})^{-\alpha}=O(t^{0}).
\end{align}

When $\xi \in (1, |y|^{\alpha}t^{-\beta})$

\begin{align}
\int_{1}^{|y|^{\alpha}t^{-\beta}}\xi^{-1/\beta}w(\xi^{-1/\beta}, \beta, 1)t^{\beta}|y|^{-d-\alpha}d\xi 
= \int_{1}^{|y|^{\alpha}t^{-\beta}}\xi^{-1/\beta}\xi^{-q/\beta}t^{\beta}|y|^{-d-\alpha}d\xi \nonumber \\
=t^{\beta}|y|^{-d-\alpha}\left((|y|^{\alpha}t^{-\beta})^{-1/\beta - q/\beta + 1} - 1 \right) 
= t^{1+q + \beta - \beta}|y|^{-d-\alpha - \alpha/\beta - q\alpha/\beta + \alpha} - t^{\beta}|y|^{-d-\alpha}.
\end{align}

Integration gives

\begin{align}\label{lq}
\int_{|y| > t^{\beta/\alpha}}t^{1 + q}|y|^{-\alpha/\beta - q\alpha/\beta - 1}dy 
=t^{1 + q}(t^{\beta/\alpha})^{-\alpha/\beta - q \alpha/\beta} = t^{0},
\end{align}

and

\begin{align}\label{l4}
\int_{|y| > t^{\beta/\alpha}}|y|^{-d-\alpha + d - 1}t^{\beta}dy = t^{\beta}(t^{\beta/\alpha})^{-\alpha}=O(t^{0}).
\end{align}

Combining (\ref{l3}), (\ref{lq}) and (\ref{l4}) gives

\begin{align}\label{l5}
\int_{\mathbb{R}^{d}}|I_{B}|dy \le C t^{0}.
\end{align}

Let's study $|I_{A}|$ case $|y| > t^{\beta/\alpha}$. Here $\xi^{-1/\beta}$ is small, so $w(\xi^{-1/\beta}, \beta, 1) \sim C$, where $C$ is a constant. 

\begin{align}
|I_{A}| \le C \int_{|y|^{\alpha}t^{-\beta}}^{\infty}\xi^{-1-1/\beta}w(\xi^{-1/\beta}, \beta, 1)t^{-d\beta/\alpha}\xi^{-d/\alpha}d\xi \nonumber \\
= C \int_{|y|^{\alpha}t^{-\beta}}^{\infty}\xi^{-1-1/\beta}t^{-\beta d/\alpha}\xi^{-d/\alpha}d\xi \nonumber \\
\le C t^{-\beta d/\alpha}(|y|^{\alpha}t^{-\beta})^{-1/\beta - d/\alpha}
\le C |y|^{-\alpha/\beta - d}t,
\end{align}

Integrating gives

\begin{align}
\int_{|y| > t^{\beta/\alpha}}|I_{A}|dy \le C\int_{|y| > t^{\beta/\alpha}}|y|^{-d -\alpha/\beta + d - 1}t dy = Ct (t^{\beta/\alpha})^{-\alpha/\beta}= O(t^{0}).
\end{align}

Let's study $|I_{A}|$, case $|y| \le t^{\beta/\alpha}$. Here we need to split the integral $I_{A}$ into 2 parts. The first one is

\begin{align}
\int_{1}^{\infty}\xi^{-d/\alpha}\xi^{-1-1/\beta}t^{-\beta d/\alpha}w(\xi^{-1/\beta}, \beta, 1)d\xi.
\end{align}

Here $\xi$ is large, so $\xi^{-1/\beta}$ is small, so $w(\xi^{-1/\beta}, \beta, 1) \sim (\xi^{-1/\beta})^{q}$, for all $q > 1$, which enables us to write

\begin{align}
\int_{1}^{\infty}\xi^{-d/\alpha}t^{-\beta d/\alpha}\xi^{-1-1/\beta}\xi^{-q/\beta}d\xi 
=t^{-\beta d/\alpha}\int_{1}^{\infty}\xi^{-d/\alpha - 1 - 1/\beta - q/\beta}d\xi \nonumber \\
\le t^{-\beta d/\alpha}\left(\lim_{K \rightarrow \infty}K^{-d/\alpha - 1/\beta - q/\beta} - 1\right) = t^{-\beta d/\alpha}.
\end{align}

Integrating gives

\begin{align}\label{l2}
\int_{|y| \le t^{\beta/\alpha}}t^{-\beta d/\alpha}|y|^{d-1}dy = t^{-\beta d/\alpha}(t^{\beta/\alpha})^{d}=O(t^{0}).
\end{align}

The second part of $I_{A}$ is

\begin{align}\label{sp}
\int_{|y|^{\alpha}t^{-\beta}}^{1}\xi^{-1-1/\beta}w(\xi^{-1/\beta}, \beta, 1)\xi^{-d/\alpha}t^{-\beta d/\alpha}d\xi.
\end{align}

Since $\xi < 1$, $\xi^{-1/\beta} > 1$, so $w(\xi^{-1/\beta}, \beta, 1) \sim (\xi^{-1/\beta})^{-1-\beta}$, so we re-write (\ref{sp}) as

\begin{align}\label{d1}
\int_{|y|^{\alpha}t^{-\beta}}^{1}\xi^{-d/\alpha}t^{-\beta d/\alpha}\xi^{-1-1/\beta +1+1/\beta}d\xi \nonumber \\
\le \int_{|y|^{\alpha}t^{-\beta}}^{1}\xi^{-d/\alpha}t^{-\beta d/\alpha}d\xi = C t^{-\beta d/\alpha}(1 - |y|^{\alpha}t^{-\beta}).
\end{align}

Integrating (\ref{d1}) in polar coordinates

\begin{align}\label{l1}
C\int_{|y| \le t^{\beta/\alpha}}|y|^{d-1}\left( t^{-\beta d/\alpha} - |y|^{\alpha}t^{-\beta}t^{-\beta d/\alpha} \right)dy \nonumber \\
\le C t^{-\beta d/\alpha}t^{\beta d/\alpha} - C t^{\beta + d\beta/\alpha - \beta - \beta d/\alpha} = C t^{0}. 
\end{align}

Combining (\ref{l1}) and (\ref{l2}) gives that for $|y| \le t^{\beta/\alpha}$

\begin{align}\label{l6}
\int_{\mathbb{R}^{d}}|I_{A}|dy \le C t^{0}.
\end{align}

Using the assumption $|f_{0}(y)|<C$ and putting together estimates (\ref{l5}) and (\ref{l6}) yields the theorem statement (\ref{b}).

\end{proof}

\begin{thm}\label{th70}
For $\alpha \in (1,2)$, $\beta \in (0,1)$
\begin{align}
\int_{\mathbb{R}^{d}}\nabla S_{\beta, 1}(t,y)f_{0}(x-y)dy = O(t^{-\beta/\alpha}).
\end{align}
\end{thm}

\begin{proof}
We differentiate $S_{\beta, 1}(t,y)$ with respect to $y$

\begin{align}
|\nabla S_{\beta, 1}(t,y)| = \Bigg|\frac{1}{\beta(2\pi)^{d}} \nabla \int_{\mathbb{R}^{d}}\int_{0}^{\infty}e^{ipy}\exp\{-a|p|^{\alpha}t^{\beta}x\}x^{-1-1/\beta}w(x^{-1/\beta}, \beta, 1)dx dp\Bigg| \nonumber \\
= \frac{1}{\beta(2\pi)^{d}} \int_{\mathbb{R}^{d}}\int_{0}^{\infty}|ip| e^{ipy}\exp\{-a|p|^{\alpha}t^{\beta}x\}x^{-1-1/\beta}w(x^{-1/\beta}, \beta, 1)dx dp \nonumber \\
=\frac{1}{\beta(2\pi)^{d}} \int_{\mathbb{R}^{d}}\int_{0}^{\infty}|p| e^{ipy}\exp\{-a|p|^{\alpha}t^{\beta}x\}x^{-1-1/\beta}w(x^{-1/\beta}, \beta, 1)dx dp.
\end{align}

Here we use the asymptotic expansions from Theorems $7.2.1$ and $7.2.2$ and Theorem $7.3.2$, which are in the appendix as equations (\ref{zero}) and (\ref{infinity}), and we use the inequality $(7.40)$ in \cite{Kolokoltsov}, which also appears in the appendix for reader's convenience, as (\ref{pro1}) and (\ref{PRO2}). For $I_{A}$ in case $|y| > t^{\beta/\alpha}$ we use that for $\xi > 1$, $\xi^{-1/\beta} < 1$ and $w(\xi^{-1/\beta}, \beta, 1) < (\xi^{-1/\beta})^{q}$, for any $q > 1$. Then

\begin{align}
| \nabla I_{A} | \le C \int_{|y|^{\alpha}t^{-\beta}}^{\infty}\xi^{-1/\alpha-1-1/\beta - d/\alpha}w(\xi^{-1/\beta}, \beta, 1)t^{-\beta/\alpha - d\beta/\alpha}d\xi \nonumber \\
\le C t^{-\beta/\alpha - d\beta/\alpha}(|y|^{\alpha}t^{-\beta})^{-1/\alpha - d/\alpha - 1/\beta - q/\beta} 
\le C t^{1 + q}|y|^{-1-q\alpha/\beta -\alpha/\beta - d}.
\end{align}

Integrating gives
 
\begin{align}\label{qw4}
\int_{|y| > t^{\beta/\alpha}}\left|\nabla I_{A}\right|dy \le C \int_{|y| > t^{\beta/\alpha}}t^{1 + q}|y|^{-d + d - 1 - 1 - q\alpha/\beta - \alpha/\beta}dy \nonumber \\
= C t^{1 + q - \beta/\alpha - q - 1}= C t^{-\beta/\alpha}.
\end{align}

Now, let's look at $I_{B}$ in case $|y| > t^{\beta/\alpha}$. Proposition \ref{pp} in the Appendix and the change of variables $\xi^{-1/\beta}=z$ yield

\begin{align}
| \nabla I_{B}| \le C \int_{0}^{|y|^{\alpha}t^{-\beta}}\xi_{-1-1/\beta}w(\xi^{-1/\beta}, \beta, 1)t^{-\beta}\xi^{-1}|y|^{-\alpha - 1}|y|^{-d-\alpha}t^{\beta}\xi d\xi \nonumber \\
\le C|y|^{-d - 1}\int_{|y|^{-\alpha/\beta}t}^{\infty}w(z, \beta, 1)dz \le C |y|^{-d-1}.
\end{align}
 
Integration gives

\begin{align}\label{qw3}
\int_{|y| > t^{\beta/\alpha}}\left|\nabla I_{B}\right|dy \le C \int_{|y| > t^{\beta/\alpha}}|y|^{-d+d-1-1}dy \le C t^{-\beta/\alpha}.
\end{align} 

Now, let's look at $I_{A}$ in case $|y| < t^{\beta/\alpha}$.

\begin{align}
| \nabla I_{A} | \le C \int_{|y|^{\alpha}t^{-\beta}}^{\infty}\xi^{-1-1/\beta}w(\xi^{-1/\beta}, \beta, 1)t^{-\beta/\alpha}\xi^{-1/\alpha}t^{-\beta d/\alpha}\xi^{-d/\alpha}d\xi.
\end{align}

We split this integral into cases $\xi \in (|y|^{\alpha}t^{-\beta}, 1)$ and $\xi \in [1, \infty)$. For $\xi \in (|y|^{\alpha}t^{-\beta}, 1)$, $\xi^{-1/\beta} > 1$ and we may use that $w(\xi^{-1/\beta}, \beta, 1) \sim (x^{-1/\beta})^{-1-\beta}$. So

\begin{align}
\left|\nabla I_{A}\right| \le C \int_{|y|^{\alpha}t^{-\beta}}^{1}\xi^{-1-1/\beta}w(\xi^{-1/\beta}, \beta, 1)t^{-\beta/\alpha}\xi^{-1/\alpha}t^{-\beta d/\alpha}\xi^{-d/\alpha} d\xi \nonumber \\
\le C \int_{|y|^{\alpha}t^{-\beta}}^{\infty}t^{-\beta/\alpha - \beta d/\alpha}\xi^{-1/\alpha - d/\alpha}d\xi \nonumber \\
= C t^{-\beta/\alpha - \beta d/\alpha} - C t^{-\beta}|y|^{-1-d + \alpha}.
\end{align}

Integration yields

\begin{align}\label{qw2}
\int_{|y| \le t^{\beta/\alpha}}\left|\nabla I_{A}\right|dy \le C \int_{|y| \le t^{\beta/\alpha}}y^{d-1}t^{-\beta/\alpha - \beta d/\alpha}dy \nonumber \\
+ C \int_{|y| < t^{\alpha/\beta}}t^{-\beta}|y|^{-1-\alpha - 1 -d + d}dy
\nonumber \\
\le C t^{-\beta/\alpha} + C t^{-\beta}(t^{\beta/\alpha})^{-1 + \alpha} \le C t^{-\beta/\alpha}.
\end{align}

As for $\xi \in [1, \infty)$, then $\xi^{-1/\beta} < 1$ and so $w(\xi^{-1/\beta}, \beta, 1) \sim C$ and

\begin{align}
\int_{1}^{\infty}\xi^{-2-1/\beta}C\xi^{-d/\alpha}t^{-\beta/\alpha - \beta d/\alpha}d\xi \nonumber \\
\le t^{-\beta/\alpha - \beta d/\alpha}C\int_{1}^{\infty}\xi^{-2-1/\beta - d/\alpha}d\xi \nonumber \\
\le t^{-\beta/\alpha - \beta d/\alpha}C \left(1 - \lim_{K \rightarrow \infty} \frac{1}{K}\right) = C t^{-\beta/\alpha - \beta d/\alpha}.
\end{align}

Integration yields

\begin{align}\label{ax}
\int_{|y| < t^{\beta/\alpha}}t^{-\beta/\alpha - \beta d/\alpha}|y|^{d-1}dy \le C t^{-\beta/\alpha - \beta d/\alpha}t^{\beta d/\alpha} = Ct^{-\beta/\alpha}.
\end{align}

Finally, $I_{B}$ in case $|y| \le t^{\beta/\alpha}$

\begin{align}
| \nabla I_{B}| \le C \int_{0}^{|y|^{\alpha}t^{-\beta}}\xi^{-1-1/\beta}w(\xi^{-1/\beta}, \beta, 1)\xi^{-1}t^{-\beta}|y|^{\alpha - 1}|y|^{-d - \alpha}t^{\beta}\xi d\xi \nonumber \\
\le C \int_{0}^{|y|^{\alpha}t^{-\beta}}\xi^{-1-1/\beta}w(\xi^{-1/\beta}, \beta, 1)|y|^{-1-d}d\xi \nonumber \\
\le C \int_{0}^{|y|^{\alpha}t^{-\beta}}\xi^{-1-1/\beta}(\xi^{-1/\beta})^{-1-\beta}|y|^{-1-d}d\xi \nonumber \\
\le C |y|^{-1-d-\alpha}t^{-\beta}.
\end{align}

Integration yields

\begin{align}\label{qw1}
\int_{|y| \le t^{\beta/\alpha}} \left|\nabla I_{B} \right| dy \le C \int_{|y| \le t^{\beta/\alpha}}|y|^{\alpha - 1 - 1  - d + d}t^{-\beta}dy \nonumber \\
= Ct^{-\beta}(t^{\beta/\alpha})^{\alpha - 1} = C t^{-\beta/\alpha}.
\end{align}

Hence (\ref{qw4}), (\ref{qw3}), (\ref{qw2}), (\ref{ax}) and (\ref{qw1}) together with the assumption $|f_{0}(y)| < C$ yield (\ref{l6}).

\end{proof}

\begin{thm}\label{th71}
For $\alpha = 2$ and assuming $|f_{0}(y)| < C$
\begin{align}\label{a2}
\int_{\mathbb{R}^{d}}S_{\beta, 1}(t,y-x)f_{0}(x)dx = O(t^{0}).
\end{align}
\end{thm}

\begin{proof}
Using (\ref{using1})

\begin{align}
\int_{\mathbb{R}^{d}}S_{\beta, 1}(t,y)dy = \int_{0}^{\infty}\int_{\mathbb{R}^{d}}(xt^{\beta})^{-d/2}\exp\{-y^{2}/(4axt^{\beta})\}x^{-1-1/\beta}w(x^{-1/\beta}, \beta, 1) dy dx \nonumber \\
=C\int_{0}^{\infty}x^{-1-1/\beta}w(x^{-1/\beta}, \beta, 1)dx.
\end{align}

We split this integral into two parts: $x \in [0,1]$ and $x \in (1, \infty)$. In the first case $x \le 1$ and $x^{-1/\beta} > 1$ so we may use $w(x^{-1/\beta}, \beta, 1) \sim (x^{-1/\beta})^{-1-\beta}$. In case $x > 1$ we may use that $w(x^{-1/\beta}, \beta, 1) \sim C$. So we obtain

\begin{align}
\int_{0}^{1}x^{-1-1/\beta}w(x^{-1/\beta}, \beta, 1)dx = \int_{0}^{1}dx = 1,
\end{align}

and

\begin{align}
\int_{1}^{\infty}x^{-1-1/\beta}w(x^{-1/\beta}, \beta, 1)dx = \int_{1}^{\infty}x^{-1-1/\beta}Cdx \nonumber \\
= C \left(\lim_{K \rightarrow \infty} K^{-1/\beta} - 1^{-1/\beta} \right) = C.
\end{align}
Together with the assumption $|f_{0}(y)| < C$, the result (\ref{a2}) follows.
\end{proof}

\begin{thm}\label{th72}
For $\alpha=2$, $\beta \in (0,1)$ and assuming $|f_{0}(y)| < C$
\begin{align}\label{a3}
\int_{\mathbb{R}^{d}}\nabla S_{\beta, 1}(t,y)f_{0}(x-y)dy = O(t^{-\beta/2}).
\end{align}
\end{thm}

\begin{proof}
We use the representation of $S_{\beta, 1}(t,y)$ in (\ref{sbone}) and write

\begin{align}
\int_{\mathbb{R}^{d}}\nabla S_{\beta, 1}(t,y)dy \nonumber \\
= \int_{0}^{\infty}x^{-3/2 - 1/\beta}t^{-\beta/2}w(x^{-1/\beta}, \beta, 1)dx.
\end{align}

We split the above integral into two: for $x \in [0,1]$ and for $x > 1$. In case $x \in [0,1]$ we use that $w(x^{-1/\beta}, \beta, 1) \sim (x^{-1/\beta})^{-1-\beta}$. In case $x > 1$ we use that $w(x^{-1/\beta}, \beta, 1) \sim C$. So we get

\begin{align}\label{sq1}
t^{-\beta/2}\int_{0}^{1}x^{-3/2 - 1/\beta}w(x^{-1/\beta}, \beta, 1)dx = t^{-\beta/2} \int_{0}^{1}x^{-1/2}dx = t^{-\beta/2}/2
\end{align}
and

\begin{align}\label{sq2}
t^{-\beta/2}\int_{1}^{\infty}x^{-3/2 - 1/\beta}w(x^{-1/\beta}, \beta, 1)dx = t^{-\beta/2}.
\end{align}

So from (\ref{sq1}) and (\ref{sq2}) and that $|f_{0}(y)| < C$ and we obtain (\ref{a3}). 

\end{proof}

\section{Smoothing properties for the linear equation}

\smallskip

Let us denote by $C^{p}(\mathbb{R}^{d})$ the space of $p$ times continuously differentiable functions. By $C^{1}_{\infty}$ we shall denote functions $f$ in $C^{1}(\mathbb{R}^{d})$ such that $f$ and $\nabla f$ are rapidly decreasing continuous functions on $\mathbb{R}^{d}$, with the sum of sup-norms of the function and all of its derivatives up to and including the order $p$ as the corresponding norm. The sup-norm is $\| f \| = \sup_{s \in [0,t]}\|f(s)\|$. Let's denote by $H^{p}_{1}$ the Sobolev space of functions with generalised derivative up to and including $p$, being in $L^{1}(\mathbb{R}^{d})$.  Here and in what follows we often identify the function $f_{0}(y)$ with the function $f(t,y)=f_{0}(y)$, $\forall t \ge 0$. 

\begin{thm}[Solution regularity]\label{th73}
For $\alpha \in (1,2]$ and $\beta \in (0, 1)$ the resolving operator 

\begin{align}\label{ro}
\Psi_{t}(f_{0})= \int_{\mathbb{R}^{d}}S_{\beta, 1}(t, y-x)f_{0}(x)dx + \int_{0}^{t}\int_{\mathbb{R}^{d}}G_{\beta}(t-s, y-x)h(s, x)dx ds 
\end{align}

satisfies the following properties

\begin{itemize}
\item $\Psi_{t} : C^{p}(\mathbb{R}^{d}) \mapsto C^{p}(\mathbb{R}^{d})$, \mbox{ and } $\sup_{s \in [0,t]}\|\Psi_{t}\|_{C^{p}} < C(t)$,
\item $\Psi_{t} : H^{p}_{1}(\mathbb{R}^{d}) \mapsto H^{p}_{1}(\mathbb{R}^{d})$ \mbox{ and } $\sup_{s \in [0,t]}\|\Psi_{t}\|_{C^{p}} < C(t)$.
\end{itemize}

\end{thm}

\begin{proof}
We look at the $C^{p}$ norm of $f_{0}$ and use Theorem \ref{th66}

\begin{align}
\| \Psi_{t}(f_{0}) \|_{C^{p}(\mathbb{R}^{d})} \le \int_{\mathbb{R}^{d}}S_{\beta, 1}(t,y)|f_{0}^{(p)}|dy + t^{\beta}\sup_{t \in [0, T]}\|h(t, \cdot)\|_{C^{p}(\mathbb{R}^{d})} \nonumber \\
\le \|f_{0}\|_{C^{p}(\mathbb{R}^{d})} + C t^{\beta} \sup_{t \in [0, T]} \|h(t, \cdot)\|_{C^{p}(\mathbb{R}^{d})}, 
\end{align}
for some constant $C > 0$. Analogously,
\begin{align}
\| f_{0} \|_{H^{p}_{1}(\mathbb{R}^{d})} \le \int_{\mathbb{R}^{d}}S_{\beta, 1}(t,y)|f_{0}^{(p)}|dy + t^{\beta}\sup_{t \in [0, T]}\|h(t, \cdot)\|_{H^{p}_{1}(\mathbb{R}^{d})} \nonumber \\
\le \|f_{0}\|_{H^{p}_{1}(\mathbb{R}^{d})} + C t^{\beta} \sup_{t \in [0, T]}\|h\|_{H^{p}_{1}(\mathbb{R}^{d})}.
\end{align}
\end{proof}
\smallskip

\begin{thm}[Solution smoothing]\label{th74}
For $\alpha \in (1,2]$ and $\beta \in (0,1)$ the resolving operator (\ref{ro}) satisfies the following smoothing properties

\begin{itemize}
\item If $f_{0}, h \in C^{p}(\mathbb{R}^{d})$ uniformly in time, then $f \in C^{p+1}(\mathbb{R}^{d})$ and for any $t \in (0,T]$
\begin{align}
\| \Psi_{t}(f_{0}) \|_{C^{p+1}(\mathbb{R}^{d})} \le t^{-\beta/\alpha}\| f_{0} \|_{C^{p}(\mathbb{R}^{d})} + Ct^{\beta - \beta/\alpha} \| h \|_{C^{p}( \mathbb{R}^{d})} 
\end{align}
\item If $f_{0}, h \in H^{p}_{1}(\mathbb{R}^{d})$ uniformly in time, then $f \in H^{p+1}_{1}(\mathbb{R}^{d})$ and for any $t \in (0,T]$
\begin{align}\label{hnorm}
\| \Psi_{t}(f_{0}) \|_{H^{p+1}_{1}(\mathbb{R}^{d})} \le t^{-\beta/\alpha}\| f_{0} \|_{H^{p}_{1}(\mathbb{R}^{d})} + C t^{\beta - \beta/\alpha}\| h \|_{H^{p}_{1}(\mathbb{R}^{d})}.
\end{align}
\end{itemize}
In particular we may choose $p = 0$, when $H^{0}_{1}(\mathbb{R}^{d})=L^{1}(\mathbb{R}^{d})$.
\end{thm}

\begin{proof}

We study the $C^{p+1}(\mathbb{R}^{d})$ norm of $\Psi_{t}(f_{0})$ and use theorems \ref{th66} and \ref{th67}

\begin{align}
\| \Psi_{t}(f_{0})\|_{C^{p+1}} \le \sup_{x \in \mathbb{R}^{d}}\int_{\mathbb{R}^{d}}\Bigg| \nabla_{x}S_{\beta, 1}(t,x-y)f_{0}^{(p)}(y) \Bigg| dy \nonumber \\
+ \sup_{x \in \mathbb{R}^{d}}\int_{0}^{t}\int_{\mathbb{R}^{d}}\Bigg|\nabla_{x}G_{\beta}(t-s, x-y)h^{(p)}_{y}(s,y) \Bigg| dy ds \nonumber \\
\le t^{-\beta/\alpha}\sup_{x \in \mathbb{R}^{d}}| f_{0}^{(p)}(x)| + \sup_{x \in \mathbb{R}^{d}}|h^{(p)}(s,x)| \int_{0}^{t}(t-s)^{\beta - \beta/\alpha - 1}ds \nonumber \\
\le t^{-\beta/\alpha}\|f_{0}\|_{C^{p}} + t^{\beta - \beta/\alpha}\| h \|_{C^{p}}.
\end{align}
The proof for (\ref{hnorm}) is analogous.
\end{proof}

Similar results apply for the non-linear equation (\ref{H}).

\section{Well-posedness}
Now we study well-posedness of the full non-linear equation (\ref{H}):

\begin{align}\label{againFDE}
D^{* \beta}_{0,t}f(t,y)=-a(-\Delta)^{-\alpha/2}f(t,y)+ H(t,y, \nabla f(t,y)),
\end{align}
with the initial condition $f(0,y)=f_{0}(y)$, and $a > 0$ is a constant. This FDE has the following mild form:
\begin{align}\label{milD2}
f(t,y)=\int_{\mathbb{R}^{d}}f_{0}(x)S_{\beta, 1}(t,y-x)dx + \int_{0}^{t} \int_{\mathbb{R}^{d}}G_{\beta}(t-s, y-x)H(s,x, \nabla f(s,x))dx ds,
\end{align}
which follows from (\ref{fh}).

\begin{lem}\label{lemmm}
Let's define by $C([0,T], C^{1}_{\infty}(\mathbb{R}^{d}))$ the space of functions $f(t,y), t \in [0,T], y \in \mathbb{R}^{d}$ such that $f(t,y)$ is continuous in $t$ and $f(t, \cdot) \in C^{1}_{\infty}(\mathbb{R}^{d})$ for all $t$. Denote by $B_{f_{0}}^{T}$ the closed convex subset of $C([0,T], C^{1}_{\infty}(\mathbb{R}^{d}))$ consisting of functions with $f(0, \cdot))=f_{0}(\cdot)=S_{0}(\cdot)$ for some given function $S_{0}$. Let us define a non-linear mapping $f \rightarrow \{\Psi_{t}(f)\}$ defined for $f \in B^{T}_{f_{0}}$:
\begin{align}\label{phide}
\Psi_{t}(f)(y)=\int_{\mathbb{R}^{d}}f_{0}(x)S_{\beta, 1}(t,y-x)dx + \int_{0}^{t} \int_{\mathbb{R}^{d}}G_{\beta}(t-s, y-x)H(s,x, \nabla f(s,x))dx ds.
\end{align}
Suppose $H(s,y,p)$ is Lipschitz in $p$ with the Lipschitz constant $L$. Let's take $f_{1}, f_{2} \in B^{T}_{f_{0}}$. Then for $K = \frac{1}{\beta - \beta/\alpha}$ and for any $t \in [0, T]$:
\begin{align}\label{lemeq}
\|\Psi_{t}^{n}(f_{1}) - \Psi_{t}^{n}(f_{2})\|_{C^{1}} \le\frac{(\beta - \beta/\alpha) L^{n}(Kt^{(\beta - \beta/\alpha)})^{n}}{n^{n\beta - n \beta/\alpha + 1}}\sup_{s \in [0,t]}\|f_{1} - f_{2}\|_{C^{1}}.
\end{align}


\end{lem}

\begin{proof}
Due to regularity estimates for $S_{\beta, 1}$ and $G_{\beta}$:

\begin{align}
\|\Psi_{t}(f_{1}) - \Psi_{t}(f_{2})\|_{C^{1}} \le C L t^{\beta - \beta/\alpha}\sup_{s \in [0,t]}\|f_{1} - f_{2}\|_{C^{1}}. 
\end{align}

and
\begin{align}
\|\Psi_{t}^{2}(f_{1}) - \Psi^{2}_{t}(f_{2})\|_{C^{1}} 
\le C^{2} L^{2} \sup_{s \in [0,t]}\|f_{1} - f_{2}\|_{C^{1}} \int_{0}^{t}(t-s)^{\beta - \beta/\alpha - 1}s^{\beta - \beta/\alpha}ds. 
\end{align}

We calculate the integral above using the change of variables $z = s/t$:

\begin{align}
\int_{0}^{t}(t-s)^{\beta - \beta/\alpha - 1}s^{\beta - \beta/\alpha}ds \nonumber \\
=\int_{0}^{1}t^{\beta - \beta/\alpha - 1}(1-z)^{\beta - \beta/\alpha - 1}z^{\beta - \beta/\alpha}t^{\beta - \beta/\alpha + 1}dz \nonumber \\
=t^{2\beta - 2\beta/\alpha}B(\beta - \beta/\alpha + 1, \beta - \beta/\alpha).
\end{align}

Now, when we estimate $\|\Psi_{t}^{3}(f_{1}) - \Psi_{t}^{3}(f_{2})\|_{C^{1}}$ we calculate

\begin{align}
\int_{0}^{t}s^{2\beta - 2 \beta/\alpha}(t-s)^{\beta  - \beta/\alpha - 1}ds \nonumber \\
=t^{\beta - \beta/\alpha - 1}\int_{0}^{1}t^{2\beta -2\beta/\alpha + 1}z^{2\beta - 2\beta/\alpha}(1-z)^{\beta - \beta/\alpha - 1}dz \nonumber \\
=t^{3(\beta - \beta/\alpha)}B(2\beta - 2 \beta/\alpha + 1, \beta - \beta/\alpha).
\end{align}

This yields

\begin{align}
\|\Psi^{3}_{t}(f_{1}) - \Psi^{3}_{t}(f_{2})\|_{C^{1}} \le C^{3} L^{3} t^{3\beta - 3 \beta/\alpha}B(2\beta - 2 \beta/\alpha + 1, \beta - \beta/\alpha) \sup_{s \in [0,t]}\|f_{1} - f_{2}\|_{C^{1}}.
\end{align}

As the inductive step, assume that the following is true for some $n \in \mathbb{N}$:

\begin{align}\label{ind}
\|\Psi^{n}_{t}(f_{1}) - \Psi^{n}_{t}(f_{2})\|_{C^{1}} \nonumber \\
\le C^{n}L^{n}t^{n\beta - n \beta/\alpha}\frac{(\Gamma(\beta - \beta/\alpha))^{n-1}\Gamma(\beta - \beta/\alpha + 1)}{\Gamma(n \beta - n \beta/\alpha + 1)}\sup_{s \in [0,t]}\|f_{1} - f_{2}\|_{C^{1}}.
\end{align}

Let's check that then (\ref{ind}) holds for $k=n+1$.

\begin{align}\label{above}
\| \Psi^{n+1}_{t}(f_{1}) - \Psi^{n+1}_{t}(f_{2})\|_{C^{1}} \nonumber \\
= \Bigg| \Bigg | \int_{0}^{t}\int_{\mathbb{R}^{d}}G_{\beta}(t-s, x-y)\left(H(s, y, \nabla \Psi^{n}_{t}(f_{1})) - H(s, y, \nabla \Psi^{n}_{t}(f_{2}))  \right) \Bigg| \Bigg |_{C^{1}} \nonumber \\
\le C L \int_{0}^{t}(t-s)^{\beta - \beta/\alpha - 1} ds \| \Psi^{n}_{t}(f_{1}) - \Psi^{n}_{t}(f_{2})\|_{C^{1}} \nonumber \\
\le C^{n+1}L^{n + 1}t^{n\beta - n\beta/\alpha} M_{n}\int_{0}^{t}(t-s)^{\beta - \beta/\alpha - 1}s^{n\beta - n \beta/\alpha}ds \sup_{s \in [0,t]}\| f_{1} - f_{2}\|_{C^{1}}\nonumber \\
\le C^{n+1}L^{n + 1}t^{n\beta - n \beta/\alpha} M_{n}B_{n} \sup_{s \in [0,t]}\| f_{1} - f_{2}\|_{C^{1}} \nonumber \\
\le C^{n+1}L^{n + 1}t^{n\beta - n \beta/\alpha} M_{n+1}\sup_{s \in [0,t]}\| f_{1} - f_{2}\|_{C^{1}},
\end{align}

where
\begin{align}\label{mm}
M_{n}=\frac{(\Gamma(\beta - \beta/\alpha))^{n - 1}\Gamma(\beta - \beta/\alpha + 1)}{\Gamma(n\beta - n \beta/\alpha + 1)},
\end{align}

$M_{n+1}$ is as in (\ref{mm}) with $n$ replaced by $n+1$, and $B_{n}$ is the Beta function

\begin{align}
B_{n}=B(n\beta - n \beta/\alpha + 1, \beta - \beta/\alpha).
\end{align}

The inequality (\ref{above}) is (\ref{ind}) with $k=n$ replaced by $k=n+1$. We have shown (\ref{ind}) is true for $k = 1$ and $k=2$. So by induction on $k$ we obtain (\ref{ind}) for any $k \in \mathbb{N}$. Using that $g(x)=x^{n}$ is a convex function for $n \in \mathbb{N}$, we have

\begin{align}
(\Gamma(\beta - \beta/\alpha))^{n} \le\frac{\Gamma(n\beta - n \beta/\alpha- n + 1)}{n^{n\beta - n \beta/\alpha - n + 1}},
\end{align}

and using Stirling's formula to obtain the quotient approximation

\begin{align}
\frac{\Gamma(n(\beta - \beta/\alpha) + B)}{\Gamma(n(\beta - \beta/\alpha)+ A )} \approx \left(n (\beta - \beta/\alpha) \right)^{B-A}.
\end{align}

Let us substitute $A = 1$ and $B = -n + 1$. Then

\begin{align}\label{WWW}
\|\Psi^{n}_{t}(f_{1}) - \Psi^{n}_{t}(f_{2})\|_{C^{1}} \nonumber \\
\le \frac{\Gamma(1 + \beta - \beta/\alpha)t^{n\beta - n \beta/\alpha}}{n^{n\beta - n \beta/\alpha + 1}(\beta - \beta/\alpha)^{n} \Gamma(\beta - \beta/\alpha)}\sup_{s \in [0,t]}\|f_{1} - f_{2}\|_{C^{1}} \nonumber \\
\le \frac{t^{n\beta - n \beta/\alpha}}{n^{n\beta - n \beta/\alpha + 1}(\beta - \beta/\alpha)^{n-1}}\sup_{s \in [0,t]}\|f_{1} - f_{2}\|_{C^{1}},
\end{align}
so (\ref{lemeq}) holds.
\end{proof}

\begin{thm}\label{th75}
Assume that
\begin{itemize}
\item $H(s,y,p)$ is Lipschitz in $p$ with the Lipschitz constant $L$ independent of $y$. 
\item $|H(s,y,0)| \le h$, for a constant $h$ independent of $y$. 
\item $f_{0}(y) \in C^{1}_{\infty}(\mathbb{R}^{d})$. 
\end{itemize}

Then the equation $(\ref{milD2})$ has a unique solution $S(t,y) \in C^{1}_{\infty}(\mathbb{R}^{d})$.
\end{thm}

\begin{proof}
Let's denote by $C([0,T], C_{\infty}^{1}(\mathbb{R}^{d}))$ and 
$B^{T}_{f_{0}}$ as in Lemma 1. Let $\Psi_{t}(f)$ be defined as in (\ref{phide}). Take $f_{1}(s,x), f_{2}(s,x) \in B^{T}_{f_{0}}$. Note that due to our choice of $f_{1}, f_{2}$, 
\begin{align}
\int_{\mathbb{R}^{d}}f_{1}(0,x)S_{\beta, 1}(t,y-x)dx = \int_{\mathbb{R}^{d}}f_{2}(0,x)S_{\beta, 1}(t,y-x)dx.
\end{align}
We would like to prove the existence and uniqueness result for all $t \le T$ and any $T \ge 0$. For this we use (\ref{lemeq}) in Lemma $1$. As $n \rightarrow \infty$, $n^{n}$ grows faster than $m^{n}$ for any fixed $m > 0$. Hence for any $t \ge 0$

\begin{align}
\|\Psi^{n}_{t}(f_{1}) - \Psi^{n}_{t}(f_{2})\|_{C^{1}} \le \frac{L^{n}(t^{\beta - \beta/\alpha}(\beta - \beta/\alpha)^{-1})^{n}(\beta - \beta/\alpha)}{n^{n\beta - n \beta/\alpha + 1}}\sup_{s \in [0,t]}\|f_{1} - f_{2}\|_{C^{1}}.
\end{align}








The sum $ \sum_{n=1}^{\infty}\frac{(t^{\beta - \beta/\alpha}(\beta - \beta/\alpha)^{-1})^{n}}{n^{ n(\beta - \beta/\alpha) + 1}}$ is convergent by the ratio test. 
%
 By Weissinger's fixed point theorem, see \cite{diethelm} Theorem D.7, $\Psi_{t}$ has a unique fixed point $f^{*}$ such that for any $f_{1} \in B^{T}_{f_{0}}$

\begin{align}
\| \Psi^{n}_{t}(f_{1}) - f^{*}\|_{C^{1}} \le \sum_{k=n}^{\infty}\frac{(t^{\beta - \beta/\alpha}(\beta - \beta/\alpha)^{-1})^{n}(\beta - \beta/\alpha)}{n^{n\beta - n \beta/\alpha + 1}}\| \Psi_{t}(f_{1}) - f_{1} \|_{C^{1}}.
\end{align}

So $S(t,y)=f^{*}$ is the solution of (\ref{milD2}) of class $C^{1}_{\infty}(\mathbb{R}^{d})$.
\end{proof}

\begin{thm}\label{76}
Assume that
\begin{itemize}
\item $H(s,y,p)$ is Lipschitz in $p$ with the Lipschitz constant $L_{1}$ independent of $y$. 
\item H is Lipschitz in $y$ independently of $p$, with a Lipschitz constant $L_{2}$ 
\begin{align}
|H(s,y_{1}, p) - H(s,y_{2}, p)| \le L_{2}|y_{1}-y_{2}|(1 + |p|)
\end{align}
\item $|H(s,y,0)| \le h$, for a constant $h$ independent of $y$. 
\item $f_{0}(y) \in C^{2}_{\infty}(\mathbb{R}^{d})$. 
\end{itemize}

Then there exists a unique solution $f^{*}(t,y)$ of the FDE equation (\ref{againFDE}) for $\beta \in (0,1)$ and $\alpha \in (1,2]$, and $f^{*}$ satisfies

\begin{align}\label{double}
ess \sup_{y}|\nabla^{2}(f^{*}(t,y))| < C.
\end{align} 
\end{thm}

\begin{proof}
First, we work with the mild form of the equation (\ref{againFDE}). Let $B^{T, 2}_{f_{0}}$ denote the subset of $B^{T}_{f_{0}}$ which is twice continuously differentiable in $y$ and with $f_{0}(y)=f_{0}(y)$, for all $y \in \mathbb{R}^{d}$. 
%
Let the mapping $\Psi_{t}$ on $B^{T, 2}_{f_{0}}$ be defined as in (\ref{phide}). Take $f_{0} \in B^{T,2}_{f_{0}}$, which continues $f_{0}(y)=S_{0}(y)$ to all $t \ge 0$. Then

\begin{align}
\|\Psi_{t}(f_{0})\|_{C^{2}} \le t^{\beta - \beta/\alpha}\sup_{s \in [0,t]}\|H(s,x,\nabla f_{0}(x))\|_{C^{1}} 
\nonumber \\
+  \|\int_{\mathbb{R}^{d}}S_{\beta, 1}(t,y-x)f_{0}(x)dx \|_{C^{2}} \nonumber \\
\le t^{\beta - \beta/\alpha}L_{1}\sup_{s \in [0,t]}\|f_{0}\|_{C^{2}} + t^{\beta - \beta/\alpha}L_{2}\sup_{s \in [0,t]}\|f_{0}\|_{C^{1}} + C t^{\beta-\beta/\alpha}\|\nabla f_{0}(x)\|_{C^{0}} + C_{3} \nonumber \\
\le Lt^{\beta - \beta/\alpha}\sup_{s \in [0,t]}\|f_{0}\|_{C^{2}} + C t^{\beta-\beta/\alpha}\sup_{s \in [0,t]}\|f_{0}(x)\|_{C^{1}} + C_{3}  \nonumber \\
\le Ct^{\beta - \beta/\alpha}\left(\sup_{s \in [0,t]}\|f_{0}\|_{C^{2}} + 1 \right) + C_{3}.
\end{align}
Iterations and induction yield
\begin{align}\label{ab}
\| \Psi_{t}^{n}(f_{0})\|_{C^{2}} \le 
C_{3}\sum_{m=1}^{n}t^{m\beta - m\beta/\alpha}K_{m} + \sum_{m=1}^{n}t^{m(\beta - \beta/\alpha)}C_{m}\left(1 + \sup_{s \in [0,t]}\|f_{0}\|_{C^{2}}\right),
\end{align}
for constants $K_{m}=B_{2}\times \cdots \times B_{m-1}$ and $C_{m} = B_{2} \times \cdots \times B_{m}$, where $B_{k}=B(k\beta - k \beta/\alpha + 1, \beta - \beta/\alpha)$, for any $k \in \mathbb{N}$. We use that for $x$ large and $y$ fixed $B(x,y) \sim \Gamma(y)x^{-y}$ to obtain that $B_{m+1} < B_{m}$, for all $m \in \mathbb{N}$ which yields that the sums $\sum_{m=1}^{n}t^{m\beta - m\beta/\alpha}K_{m}$ and $\sum_{m=1}^{n}t^{m\beta - m\beta/\alpha}C_{m}$ are convergent as $n \rightarrow \infty$. So for some constants $A_{1}, A_{2}, C_{f_{0}} > 0$,
\begin{align}
\|\Psi^{n}_{y} f_{0} \|_{C^{2}} < A_{1} + A_{2}\sup_{s \in [0,t]}\|f_{0}\|_{C^{2}} < C_{f_{0}}.
\end{align}

Hence, $\forall n \in \mathbb{N}$

\begin{align}\label{boun}
\| \nabla(\Psi^{n}_{t} f_{0}) \|_{Lip} < C_{f_{0}}.
\end{align}

It is clear that if $g_{n}(x) \rightarrow g(x)$, for all $x \in \mathbb{R}^{d}$, for continuous functions $g_{n}, g$ such that $\|g_{n}\|_{Lip} \le C$ $\forall n \in \mathbb{N}$, then $\| g \|_{Lip} \le C$. Hence, with $g_{n}=\nabla (\Psi^{n}_{t}f_{0})$, we obtain

\begin{align}
\|\lim_{n \rightarrow \infty} \nabla(\Psi^{n}_{t}f_{0})\|_{Lip} < C_{f_{0}}.
\end{align}

By Rademacher's theorem it follows that $\lim_{n \rightarrow \infty}(\nabla^{2}(\Psi^{n}_{t}(f_{1}))$ exists a.e. We invite the reader to see \cite{evans} for the Rademacher's theorem and its proof. From the previous theorem $\lim_{n \rightarrow \infty} \Psi^{n}_{t}= f^{*}$. The limit is understood in the sense of convergence in $C^{1}_{\infty}(\mathbb{R}^{d})$. 
Therefore $f^{*}$ satisfies (\ref{double}). 

\end{proof}

\begin{thm}\label{th77}

Assume that
\begin{itemize}
\item $H(s,y,p)$ is Lipschitz in $p$ with the Lipschitz constant $L$ independent of $y$. 
\item H is Lipschitz in $y$ independently of $p$, with a Lipschitz constant $L_{2}$ 
\begin{align}
|H(s,y_{1}, p) - H(s,y_{2}, p)| \le L_{2}|y_{1}-y_{2}|(1 + |p|)
\end{align}
\item $|H(s,y,0)| \le h$, for a constant $h$ independent of $y$. 
\item $f_{0}(y) \in C^{2}_{\infty}(\mathbb{R}^{d})$. 
\end{itemize}

Then a solution to the mild form

\begin{align}\label{mi}
f(t,y)=\int_{\mathbb{R}^{d}}S_{\beta, 1}(t, x-y)f_{0}(y)dy + \int_{0}^{t}\int_{\mathbb{R}^{d}}G_{\beta}(t-s, x-y)H(s, y,\nabla f(s,y))ds dy
\end{align}
which satisfies (\ref{double}), is a classical solution to 

\begin{align}\label{clas}
D^{* \beta}_{0,t}f(t,y) = -(-\Delta)^{\alpha/2}f(t,y) + H(t, y, \nabla f(t,y)).
\end{align}

\end{thm}

\begin{proof}

Let us define $\Psi_{t}(f)$ as in (\ref{phide}).
%
Firstly, by Diethelm 

\begin{align}
\hat{f}(t,p)=\hat{f}_{0}(p)E_{\beta,1}(-a|p|^{\alpha}t^{\beta}) + \int_{0}^{t}(t-s)^{\beta - 1}E_{\beta, \beta}(-a(t-s)^{\beta}|p|^{\alpha}))\hat{H}(s,y, p)ds,
\end{align}
is equivalent to 
\begin{align}\label{FT}
D^{* \beta}_{0,t}\hat{f}(t,p)=-a|p|^{\alpha}\hat{f}(t,p) + \hat{H}(t,y, \nabla f(t,y)),
\end{align}

which  in turn is equivalent to (\ref{clas}) as its Fourier transform. Also, (\ref{mi}) is equivalent to (\ref{fh}) as its inverse Fourier transform. Therefore (\ref{mi}) is equivalent to (\ref{clas}). We may carry out these equivalence procedures when $D^{* \beta}_{0,t}\Psi_{t}(f)$ and $-(-\Delta)^{\alpha/2}f$ are defined for $f$ satisfying (\ref{double}). 

Due to theorem assumptions:
\begin{align}\label{ash}
|H(s,y, \nabla f(t, \cdot))| \le h + L \sup_{s \in [0,t]}\|\nabla f(t, \cdot)\|_{C^{1}(\mathbb{R}^{d})} < \infty.
\end{align}

So

\begin{align}\label{R1}
D^{* \beta}_{0,t}\left(\int_{0}^{t}\int_{\mathbb{R}^{d}}G_{\beta}(t,y)H(s, y,\nabla f(t,y)) dy ds \right) \nonumber \\
\le \frac{C}{\Gamma[1-\beta]}\int_{0}^{t}(t-s)^{-\beta}s^{\beta}ds 
\le C_{1} \int_{0}^{1}(t-tz)^{-\beta}\beta(tz)^{\beta-1}tdz \nonumber \\
\le C_{1}\beta \int_{0}^{1}(1-z)^{1-\beta - 1}z^{\beta - 1}dz
\le C_{1}\beta B(1-\beta, \beta) < \infty.
\end{align}

Similarly

\begin{align}\label{R2}
D^{* \beta}_{0,t}\int_{\mathbb{R}^{d}}S_{\beta, 1}(t,x-y)f_{0}(y)dy\nonumber \\
\end{align}

exists when $f_{0}(y)$ gives dependence of $\int_{\mathbb{R}^{d}}S_{\beta, 1}(t,x-y)f_{0}(y)dy$ on $t$ such as $t^{k}$, where $k > -1$. This is because
\begin{align}
\int_{0}^{t}(t-s)^{-\beta}\left(\frac{d}{ds}s^{k}\right)ds \nonumber \\
=t^{k + 1 - \beta} \int_{0}^{1}(1-z)^{-\beta}z^{k-1}dz =t^{k + 1 - \beta} B(1-\beta, k + 1),
\end{align}

where for any $\beta \in (0,1)$ the Beta function $B(1-\beta, k + 1)$ is defined for $k + 1 > 0$. Hence, due to (\ref{ash}), (\ref{R1}) and  (\ref{R2}), $D^{* \beta}_{0,t}\Psi_{t}(f)$ is defined for the solution $f$ for (\ref{clas}). For $f$ satisfying (\ref{double}), when $\alpha \in (1, 2]$, $-(-\Delta)^{\alpha/2}f$ is defined. Now, let us study the solution $f^{*}(t,y)$

\begin{align}
f^{*}(t,y)=\int_{0}^{t}\int_{\mathbb{R}^{d}}G_{\beta}(t-s, y-x)H(s,x,\nabla f^{*}(t,x))dx ds  \nonumber \\
+ \int_{\mathbb{R}^{d}}S_{\beta, 1}(t, y-x)f_{0}(x)dx.
\end{align}

Differentiating twice w.r.t. $y$ gives:
\begin{align}\label{fin}
\nabla^{2}\int_{0}^{t}\int_{\mathbb{R}^{d}}G_{\beta}(t-s, y-x)H(s,x,\nabla f^{*}(t,x)) dx ds \nonumber \\
= \int_{0}^{t}\int_{\mathbb{R}^{d}}\nabla_{y} G_{\beta}(t-s, y-x)\nabla_{x} H(s,x,\nabla f^{*}(s,x))dx ds. 
\end{align}

From the representations of $G_{\beta}(t,y)$ and $\nabla G_{\beta}(t,y)$ used in theorems (\ref{th66}), (\ref{th67}) it is clear that $\nabla G_{\beta}(t,y)$ exists and is continuous in $t$ and in $y$. From theorem \ref{th75} we know $\nabla f^{*}$ exists and is Lipschitz continuous. Since we assumed $H$ to be Lipschitz, it follows from Rademacher's theorem that $\nabla_{x} H(s,x,\nabla f^{*}(s,x))$ is almost everywhere defined and bounded. Hence (\ref{fin}) represents a continuous function in $y$ and in $t$. Since $f_{0} \in C^{2}_{\infty}(\mathbb{R}^{d})$ and due to theorem (\ref{th69}) 
\begin{align}
\nabla^{2}\int_{\mathbb{R}^{d}}S_{\beta, 1}(t, y-x)f_{0}(x)dx \nonumber \\
= \int_{\mathbb{R}^{d}}S_{\beta, 1}(t, x)\nabla^{2} f_{0}(y-x) dx < \infty.
\end{align}

Thus, $\nabla^{2}f^{*}(t,y)$ exists and so $f^{*}(t,y) \in C^{2}_{\infty}(\mathbb{R}^{d})$. This completes the necessary requirements for the solution of the mild form (\ref{mi}) to be the solution of (\ref{clas}) of class $C^{2}_{\infty}(\mathbb{R}^{d})$, i.e. a solution in the classical sense.
\end{proof}


\section{Appendix}

Let us recall the asymptotic properties of stable densities defined in (\ref{symm})

\begin{align}\label{heretoo}
g(y, \alpha, \sigma)=\frac{1}{(2\pi)^{d}}\int_{\mathbb{R}^{d}}\exp\{-\sigma |p|^{\alpha}\}e^{-ipy}dp,
\end{align}
see \cite{Kolokoltsov} for details. For $|y|/\sigma^{1/\alpha} \rightarrow 0$ the following asymptotic expansion for $g$ holds

\begin{align}\label{zero}
g(y, \alpha, \sigma) \sim \frac{|S^{d-2}|}{(2\pi \sigma^{1/\alpha})^{d}}\sum_{k=0}^{\infty}\frac{(-1)^{k}}{(2k)!}a_{k}\left(\frac{|y|}{\sigma^{1/\alpha}}\right)^{2k},
\end{align}

where

\begin{align}
a_{k}=\alpha^{-1}\Gamma\left(\frac{2k + d}{\alpha}\right)B\left(k + \frac{1}{2}, \frac{d-1}{2}\right),
\end{align}

where 
\begin{align}
B(q,p)=\int_{0}^{1}x^{p-1}(1-x)^{q}dx = \frac{\Gamma(p)\Gamma(q)}{\Gamma(p + q)}
\end{align}
is the Beta function, and

\begin{align}
|S^{d-2}|= 2 \frac{\pi^{(d-1)/2}}{\Gamma(\frac{d-1}{2})}
\end{align}

and $|S^{0}| = 2$, see \cite{Kolokoltsov} for the proof.

\smallskip

For $|y|/\sigma^{1/\alpha} \rightarrow \infty$ the following asymptotic expansion holds

\begin{align}\label{infinity}
g(y; \alpha, \sigma) \sim (2\pi)^{-(d+1)/2}\frac{2}{|y|^{d}}\sum_{k=1}^{\infty}\frac{a_{k}}{k!}(\sigma |y|^{-\alpha})^{k}
\end{align}

where

\begin{align}
a_{k}=(-1)^{k+1}\sin\left(\frac{k \pi \alpha}{2}\right)\int_{0}^{\infty}\xi^{\alpha k + (d-1)/2}W_{0, \frac{d}{2} - 1}(2\xi)d\xi
\end{align}

and $W_{0,n}(z)$ is the Whittaker function

\begin{align}
W_{0,n}(z)=\frac{e^{-z/2}}{\Gamma(n + 1/2)}\int_{0}^{\infty}[t(1+t/z)]^{n-1/2}e^{-t}dt,
\end{align}

see \cite{Kolokoltsov} for the proof.

\medskip

In case $d=1$ the stable density function $w(x, \beta, 1)$ defined in (\ref{skew}) is infinitely smooth for $x=0$ and $w(x, \beta, 1)=0$ for $x < 0$. Hence $w$ grows at $0$ slower than any power. This gives rise to the inequalities such as $w(x, \beta, 1) < C_{q}x^{q-1}$ for any $q > 1$, for $x < 1$. The property $w(x) \sim x^{-1-\beta}$ for $x >> 1$, may be found for example in \cite{Kolokoltsov}. This may be deduced from the asymptotic expansions in equations $7.7$ and $7.9$ in \cite{Kolokoltsov} with $\gamma=1$.

\smallskip

The following result is part of the proposition 7.3.2 from \cite{Kolokoltsov}:
\begin{prop}\label{pp}
Let

\begin{align}
\phi(y, \alpha, \beta, \sigma)=\frac{1}{(2\pi)^{d}}\int_{\mathbb{R}^{d}}|p|^{\beta}exp\{-i(p,y) - \sigma|p|^{-\alpha}\}dp,
\end{align}
so that
\begin{align}
\frac{\partial \phi}{\partial \beta}(y, \alpha, \beta, \sigma) = \frac{1}{(2\pi)^{d}}\int_{\mathbb{R}^{d}}|p|^{\beta}log|p|exp\{-i(p,y) - \sigma|p|^{-\alpha}\}dp.
\end{align}

Then if $\frac{|y|}{\sigma^{1/\alpha}} \le K$

\begin{align}\label{pro1}
|\phi(y, \alpha, \beta, \sigma)| \le c \sigma^{-\beta/\alpha}g(y, \alpha, \sigma)
\end{align}
and if  $\frac{|y|}{\sigma^{1/\alpha}} > K$

\begin{align}\label{PRO2}
|\phi(y, \alpha, \beta, \sigma)| \le  c \sigma^{-1}|y|^{\alpha - \beta}g(y, \alpha, \sigma),
\end{align}

where $g$ is as in (\ref{heretoo}) and (\ref{symm}).

\end{prop}





\section*{Acknowledgements}

The authors would like to thank BCAM, G. Pagnini and E. Scalas for organizing the BCAM FCPNLO 2013 workshop where this work was discussed. This work is supported by the IPI RAN grants RFBR 11-01-12026 and 12-07-00115,
and by the grant 4402 of the Ministry of Education and Science of Russia, and by EPSRC grant EP/HO23364/1 through MASDOC, University of Warwick, UK.

\bibliographystyle{alpha}
\bibliography{WellposednessA2}

\end{document}